\makeindex
\documentclass[graybox]{svmult}

\usepackage{makeidx}         % allows index generation
\usepackage{graphicx}        % standard LaTeX graphics tool
\usepackage{multicol}        % used for the two-column index
\usepackage[bottom]{footmisc}% places footnotes at page bottom
\usepackage{ucs}
\usepackage[cp1250]{inputenc}
\usepackage{color}
\usepackage{dsfont}
\usepackage{amsfonts}
\usepackage{amssymb}
\usepackage{amsmath}
\usepackage{textcomp}
\usepackage{verbatim}
\usepackage{color}

\makeindex             % used for the subject index
\begin{document}

\def\N{{\mathbb N}}
\def\R{{\mathbb R}}

\title*{Continuous selections of multivalued mappings}
\author{Du\v san Repov\v s and Pavel V. Semenov}
\institute{Du\v san Repov\v s
\at
Faculty of Education and
Faculty of Mathematics and Physics,
University of Ljubljana,
P. O. Box 2964,
Ljubljana, Slovenia 1001\\
\email{dusan.repovs@guest.arnes.si}
\and
Pavel V. Semenov
\at
Department of Mathematics,
Moscow City Pedagogical University,
2-nd Sel'sko\-khozyast\-vennyi pr.\,4,
Moscow,\,Russia 129226\\
\email{pavels@orc.ru}}

\maketitle

\abstract{
This survey covers
in our opinion
the
most important results in the
 theory of continuous
selections of multivalued mappings
(approximately)
from
2002 through 2012.
It extends and
continues our previous 
such
survey which appeared in 
{\it Recent Progress in General Topology, II} which
was published in 2002.
In comparison, our present survey
considers more restricted
and specific areas of mathematics. We remark
that we do not consider 
the theory of selectors
(i.e. continuous choices of elements from subsets of  topological spaces)
since this topics is covered by another survey in this volume.
}

\section{Preliminaries}
\label{sec:1}

A 
{\it selection}\index{selection}
of a given 
multivalued mapping\index{multivalued mapping} 
$F:X \to Y$
with nonempty values 
$F(x) \neq \emptyset,$
for every
$x \in X,$
is a mapping $\Phi:X \to Y$ 
(in general, also multivalued) 
which for every $x \in X,$ selects a nonempty subset $\Phi(x) \subset F(x)$.
When all $\Phi(x)$ are singletons,
a selection is called 
{\it singlevalued} %\index{singlevalued selection} 
and is identified with the usual singlevalued mapping
$f:X \to Y, \,\,\{f(x)\}=\Phi(x).$ 
As a rule, we shall use small letters $f, g, h, \phi, \psi,...$ 
for singlevalued mappings and 
capital letters $F, G, H, \Phi, \Psi,...$ for multivalued mappings.

There exist a great number of theorems on existence of selections in the category of topological spaces and their continuous (in various senses) mappings.
However, the citation index of one
of them is by an order of magnitude
higher than for any other one: this is the Michael selection theorem for 
convexvalued mappings\index{convexvalued mapping}
\cite[Theorem 3.2'',\, (a)$\Rightarrow$(b)]
{Sel-1}:

\begin{theorem} \label{ConvSel}
A multivalued mapping $F:X \to B$
admits  a continuous single\-valued selection, provided
that the following conditions are satisfied:
\begin{description}
\item{(1)} $X$ is a paracompact space;
\item{(2)} $B$ is a Banach space;
\item{(3)} $F$ is a lower semicontinuous (LSC) mapping;
\item{(4)} For every $x \in X$,
$F(x)$ is a nonempty convex subset of $B$; and
\item{(5)} For every $x \in X$,
$F(x)$ is a closed subset of $B$.
\end{description}
\end{theorem}

Moreover, the reverse implication $(b)\Rightarrow(a)$ in \cite[Theorem 3.2'']{Sel-1} states that a
$T_1$-space $X$ is paracompact whenever each multivalued mapping $F:X \to B$ with properties $(2)-(5)$ above, admits  a continuous singlevalued selection.

If one identifies a multivalued mapping $F: X \to Y$ with its graph $\Gamma_F \subset X \times Y$ then the {\it lower semicontinuity}\index{lower semicontinuous mapping} 
(LSC) of $F$ means exactly the openess of the restriction $\pi_1 | _{\Gamma_F}: \Gamma_F \to X $ of the projection $\pi_1: X \times Y \to X $ onto the first factor.
In more direct terms, lower semicontinuity of a multivalued
mapping $F: X \to Y$ between  topological spaces $X$ and $Y$ means
that the (large) preimage
$
F^{-1}(U) = \{x \in X\,\,: F(x) \cap U \ne \emptyset\}
$
of any open set $U \subset Y$ is an open subset of the domain $X$.
Applying the Axiom of Choice,  we obtain:

\begin{lemma} \label{LSC}
The following statements are equivalent:
\begin{description}
\item{(1)} $F: X \to Y$ is a lower semicontinuous mapping; 
\item{(2)} For each $(x;y) \in\Gamma_F$ and
each open
neighborhood $U(y)$ there exists
a local singlevalued (not necessarily continuous) selection of $F$, say
$s:x' \mapsto s(x') \in F(x') \cap U(y),$ defined on some open
neighborhood $V(x)$.
\end{description}
\end{lemma}

Therefore, the notion of lower semicontinuity is by definition related to the
notion of selection. Symmetrically, if the (large) preimage
$
F^{-1}(A) = \{x \in X\,\,: F(x) \cap A \ne \emptyset\}
$
of any {\it closed} set $A \subset Y$ is a closed subset of the domain $X$, then the mapping
$F:X \to Y$ is said to be 
{\it upper semicontinuous}\index{upper semicontinuous mapping}
(USC). Note that
a
more useful definition of upper semicontinuity of $F$ is that the
(small) preimage
$
F_{-1}(U) = \{x \in X\,\,: F(x) \subset U \}
$
of any open $U \subset Y$ is an open subset of the domain $X$.

Let us now reformulate the other three principal Michael's theorems on selections.

\begin{theorem} [\cite{Sel-2}]
\label{ZeroDim}
A multivalued mapping $F:X \to Y$
admits  a continuous singlevalued selection, provided
that the following conditions are satisfied:
\begin{description}
\item{(1)} $X$ is a zero-dimensional (in {\rm dim-}sense) paracompact space;
\item{(2)} $Y$ is a completely metrizable space;
\item{(3)} $F$ is a LSC mapping; and
\item{(4)} For every $x \in X$,
$F(x)$ is a closed subset of $Y$.
\end{description}
\end{theorem}

\begin{theorem} [\cite{MultSel}] \label{CompSel}
A multivalued mapping $F:X \to Y$
admits a compactvalued USC selection $H:X \to Y$, which in turn,
admits a compactvalued LSC selection $G:X \to Y$
(i.e. $G(x) \subset H(x) \subset F(x),\,\,x \in X$), provided
that the following conditions are satisfied:
\begin{description}
\item{(1)} $X$ is a paracompact space;
\item{(2)} $Y$ is a completely metrizable space;
\item{(3)} $F$ is a LSC mapping; and
\item{(4)} For every $x \in X$, $F(x)$ is a closed subset of $Y$.
\end{description}
\end{theorem}

\begin{theorem} [\cite{Sel-2}] \label{FinDim}
Let $n \in \N$. A multivalued mapping $F:X \to Y$
admits a continuous singlevalued selection provided
that the following conditions are satisfied:
\begin{description}
\item{(1)} $X$ is a paracompact space with $dim X \leq n+1$;
\item{(2)} $Y$ is a completely metrizable space;
\item{(3)} $F$ is a LSC mapping;
\item{(4)} For every $x \in X$, $F(x)$ is an $n$-connected subset of $Y$;
and
\item{(5)} The family of values $\{F(x)\}_{x \in X}$ is equi-locally $n$-connected.
\end{description}
\end{theorem}

A resulting selection of a given multivalued mapping $F$ is
practically always constructed as a uniform limit of some sequence of approximate selections.
A typical difficult situation arises with the limit point (or the limit subset). 
Such a limit point (or a subset) can easily end up in the boundary of the set $F(x)$, rather than
in the set
$F(x)$,\,\,if one does not pay attention to a more careful
construction of the uniform Cauchy sequence of approximate selections.

In general, for an arbitrary Banach space $B,$ there exists a LSC mapping $F:[0;1] \rightarrow B$ with convex (nonclosed) values and without any continuous singlevalued selections (cf. \cite[Example 6.2 ]{Sel-1} or \cite[Theorem 6.1]{RS}). On the other hand, every
convexvalued LSC mapping of a metrizable domain into a separable
Banach space admits a singlevalued selection,
provided that all values are
finite-dimensional
\cite[Theorem 3.1''']{Sel-1}.
Another kind of omission of closedness was suggested in \cite{ChobGen,
CountSel}. It turns out that such omission can be made over a
$\sigma$-discrete subset of the domain.

An alternative to pointwise omission of closedness is to consider
some uniform versions of such omission. Namely, one can consider
closedness in a fixed subset $Y \subset B$ instead of closedness in
the entire Banach space $B$. Due to a deep result of
van Mill, Pelant and Pol \cite{MPP}, existence of selections under such
assumption implies that $Y$ must be completely metrizable, or in
other words, a $G_{\delta}$-subset of $B$.

Due to the Aleksandrov Theorem, each of Theorems 2-4 remains valid
under a replacement of the entire completely metrizable range space by any of its $G_{\delta}$-subsets. However,  what happens with Theorem 1 under such a substitution? What can one
(informally)
say concerning the links between the metric structure and the convex structure induced on a $G_{\delta}$-subset from the entire Banach space?

Thus, during the last two decades one of the most intriguing questions in the selection theory was the following problem:

\begin{problem} [\cite{MOpPr}]\label{G-dProb} Let $Y$ be a convex $G_{\delta}$-subset of a Banach space $B$.
Does then every LSC mapping $F : X \to Y$ of a paracompact
space $X$ with nonempty convex closed values  into $Y$
have a continuous singlevalued selection?
\end{problem}

In the next section we shall present some (partial) affirmative answers, as well as the counterexample
of Filippov \cite{Filli,Fil}.

\section{Solution of the $G_{\delta}$-problem}
\label{sec:2}

Summarizing the results below, the answer to the 
$G_{\delta}$-problem\index{$G_{\delta}$-problem}
is affirmative for domains which are "almost" \, finite-dimensional,
whereas the answer is negative for domains which are essentially infinite-dimensional, for example, for domains which contain a copy of the Hilbert cube.

For finite-dimensional domains $X$, the
$G_{\delta}$-problem has an
affirmative solution simply because the family of convex closed
subsets of a Banach space is 
$ELC^{n}$\index{$ELC^{n}$-family} 
and every convex set is $C^{n}$ for every
$n \in \mathbb{N}$. Hence Theorem 4 can be applied. For a finite-dimensional range $B$ and moreover, for
all finite-dimensional values closed in $Y \subset B$ , the problem is
also trivial, because one can use the compactvalued selection
Theorem 3 and the fact that the closed convex hull of a
finite-dimensional compact space coincides with its convex hull.

As for ways of uniform omission of closedness in the range space let us first
consider the simplest case when $Y=G$ is a unique open subset of a Banach space $B$.
Separately we extract the following well-known folklore result
(it probably first appeared in an implicit form in Corson and Lindenstrauss \cite{CorLi}).

\begin{lemma} \label{LocPrinc}
{\bf (Localization Principle)}
Suppose that a convexvalued mapping $F: X
\rightarrow Y$ of a paracompact domain $X$ into a topological
vector space $Y$ admits a singlevalued continuous selection over
each member of some open covering $\omega$ of the
domain. Then $F$
admits a global singlevalued continuous selection.
\end{lemma}

Taking for any $x\in X$ and $y \in F(x) \subset G$,
an arbitrary open ball $D$, centered at $y$, such that the closure
$Clos(D)$ is a subset of $G$, and invoking the 
Localization Principle\index{localization principle}
we obtain:

\begin{lemma} \label{SelOpenSet}
Given any paracompact space $X$ and any
open subset $G$ of a Banach space $B$, every LSC mapping $F: X
\rightarrow G$ with nonempty convex values admits a singlevalued
continuous selection, whenever all values $F(x)$ are closed in $G$.
\end{lemma}

Somewhat different approach can be obtained using the following:

\begin{lemma} \label{SelOpenSet_1}
For any compact subset $K$ of a convex closed (in $G$)
subset $C$ of an open subset $G$ of a Banach space $B$, the closed (in $B$)
convex hull $Clos(conv(K))$ also
lies in $C$.
\end{lemma}

Thus, as it was pointed out in \cite{MOpPr, RS-2},
one can affirmatively resolve the $G_{\delta}$-problem for
an arbitrary
intersection
of countably many open 
{\it convex }
subsets of $B$.

\begin{lemma} \label{SelOpenSet_2}
Let $\{G_{n}\}, n \in \mathbb{N}$,
be a sequence of
open convex subsets of a Banach space and 
$F: X \rightarrow Y
= \bigcap _{n}G_{n}$ 
a LSC mapping of a paracompact space $X$
with nonempty convex values. Then $F$ admits a singlevalued
continuous selection,
whenever all values $F(x)$ are closed in $Y$.
\end{lemma}

In fact, it suffices to pick a compactvalued LSC selection $H : X \to Y$ of the
mapping $F$ (cf.  Theorem 3).
Then the multivalued mapping $ Clos(conv(H)) : x \mapsto Clos(conv(H(x)))$ is a selection
of the given mapping $F$ and it remains to apply Theorem 1 to the LSC
mapping $Clos(conv(H))$.

Michael and Namioka \cite{MiNam} characterized those convex $G_{\delta}$-subsets $Y \subset B$ which are stable with respect to taking closed convex hulls of compact subsets. Note that they  essentially used the construction of Filippov's counterexample \cite{Filli,Fil}.

\begin{theorem} [\cite{MiNam}]\label{Mich+Nam}
Let $Y \subset B$ be a convex $G_{\delta}$-subset of a Banach space $B$. Then the following statements  are equivalent:
\begin{description}
\item{(1)} If $K \subset Y$ is compact then so is the closed (in $Y$) convex hull of  $K$;
\item{(2)} For any paracompact space $X,$
each LSC mapping $F:X \to Y$ with convex closed (in $Y$) values admits a continuous singlevalued selection; 
\item{(3)} Same as (2) but with $X$ assumed to be  compact
and
metrizable.
\end{description}
\end{theorem}

Moreover, they observed that Theorem 5 remains valid for nonconvex $Y$, provided that $(1)$ is modified by also requiring that $K \subset C$ for some  closed (in $Y$) convex subset $C \subset Y$. Hence,
the equivalence $(2) \Leftrightarrow (3)$ of Theorem 5 holds for any $G_{\delta}$-subset $Y$ of a Banach space.

Returning to the restrictions for domains, recall that
Gutev  \cite{Gu-1} affirmatively resolved the $G_{\delta}$-problem for domains $X$ which are either a countably dimensional metric space or a strongly countably dimensional paracompact space.
In fact, he proved that in both cases under the hypotheses of the problem,
the existence of a singlevalued continuous selection is equivalent
to the existence of a compactvalued USC selection. The latter statement is true, because each domain of such type can be represented as the image of some zero-dimensional
paracompact space under some closed surjection with all  preimages
of points being finite.

In 2002, Gutev and Valov \cite{GV} obtained a
positive answer
for domains with the so-called 
$C$-property\index{$C$-property}.
They introduced a
certain enlargement of the original mapping $F$. Roughly speaking,
they defined $W_{n}(x)$ as the set of all $y \in Y = \bigcap_{n
\in N} G_{n}$ which are closer to $F(x)$ than to $B \setminus G_{n}$. It
turns out that each of the mappings $W_{n}$ has an open graph and
all of its values are contractible.

Applying the selection theorem
of Uspenskii \cite{U} for $C$-domains,
one can first find selections for
each $W_{n}$ and then for the pointwise intersections $\bigcap
W_{n}(x)$ (for details cf.  \cite{GV}). Their technique properly works even for
arbitrary (nonconvex) $G_{\delta}$-subsets $Y \subset B$.
Unfortunately, such a method does not work
outside the class of $C$-domains, because Uspenskii's theorem
gives a
{\it characterization} of the $C$-property.

This was the reason why we stated the following problem in our previous survey:

\begin{problem}[\cite{RS-2}] \label{CharC-dom}
Are the following statements
equivalent:
\begin{description}
\item{(1)} $X$ is a $C$-space.
\item{(2)} Each LSC mapping $F:X \to Y$ to a $G_{\delta}$-subset of $Y$ of a Banach space $B$
with convex closed (in $Y$) values admits a continuous singlevalued selection.
\end{description}
\end{problem}

Karassev has resolved this problem for weakly infinite-dimensional compact domains.

\begin{theorem}[\cite{Kar}] \label{Karas}
Let $X$ be a compact Hausdorff space and suppose that property $(2)$ above holds. Then $X$ is weakly infinite-dimensional.
\end{theorem}

Observe that the (non)coincidence of the classes of $C$-spaces and weakly infinite-dimensional spaces is one of the oldest and still unsolved problems in dimension theory. The advantage of compact spaces is that 
in this case
there is a set of various criteria for weak infinite-dimensionality.
In particular, Karassev \cite{Kar} used the fact that a compact space $X$ is weakly infinite-dimensional if and only if for any mapping $f:X \to Q$ to the Hilbert cube there exists a mapping $g: X \to Q$ such that $f(x) \not=g(x),$ for all $x \in X$.

Ending with the affirmative answers, let us recall that in our previous survey we
directly suggested (cf.  p. 427 in \cite{RS-2}) the area for finding a counterexample. In fact, having Lemma 5, one needs to find a convex $G_{\delta}$-subset $Y$ of a Banach space $B$, such that $Y$ is not an intersection of countably many open convex sets. Such a situation in fact,
appeared in measure theory: for example in the compactum $P[0,1]$ of probability measures on the segment
$[0,1]$ such is the convex complement of any absolutely continuous measure. To extract the main idea of  Filippov's construction we introduce a temporary notion.

\begin{definition} \label{WFil}
A convex compact subset $K$ of a Fr\'{e}chet space $B$ has the 
{\it Weizs{\"a}cker-Filippov property}\index{Weizs{\"a}cker-Filippov property}
(WF-property) if there exists:
\begin{description}
\item{(1)} A proper convex $G_{\delta}$-subset $Y \subset K$ which contains
the set $extr(K)$ of all extreme points of $K$; and
\item{(2)} A LSC convexvalued mapping $R:K \to K$ such that $R(x) \cap extr(K) \not= \emptyset,\,\,x \in K$ and $R(conv A)=conv A,$ for any finite subset $A \subset extr(K)$.
\end{description}
\end{definition}

\begin{lemma} \label{CountEx}
If $K \subset B$ has the WF-property then the mapping $F: K \to Y$ defined by
$F(x) = Clos_{Y}(R(x) \cap Y), \quad x \in K$,
is a counterexample to the $G_{\delta}$-problem.
\end{lemma}

\begin{proof} All values $F(x)$ are nonempty because $R(x) \cap extr(K) \not= \emptyset$ and $extr(K) \subset Y$. Clearly, $F(x)$ are convex closed (in $Y$) sets. The mapping $x \mapsto R(x) \cap Y $ is LSC because $R$ is LSC and $Y$ is dense in $K$. Hence $F: K \to Y$ is LSC because pointwise-closure operator preserves lower semicontinuity.

Suppose to the contrary, that $f:K \to Y$ is a singlevalued continuous selection of $F$. Then $f(x)=x,$ provided that $x$ is an extreme point. Moreover, if $x,y \in extr(K)$ then  $f([x,y]) \subset F([x;y]) = [x,y]$ and $f([x,y]) \supset [x,y]$,  
since
$f(x)=x$, $f(y)=y,$ and because of
the continuity of $f$.
Similarly, $f(conv\{x,y,z\}) = conv\{x,y,z\}$ for each extreme points $x,y,z,$ and so on. Hence,
$f(conv(extr(K))) = conv(extr(K)) \subset K$ is a dense subset of $K$ and $f(K)$ is also dense in $K$. However, $f(K)$ is compact since it is the image of a
compact set $K$ under the continuous mapping $f$.
Therefore $f(K)=K$ which contradicts with $f(K) \subset Y$ and $Y \not= K$. \qed \end{proof}

\begin{theorem} \label{ProbMeas}
\begin{description}
\item{(1)} The space $P[0;1]$ of all probability measures on $[0;1]$ has the WF-property.
\item{(2)} In any Banach space there exists a convex compact subset $K$ with the
WF-property.
\end{description}
\end{theorem}

\begin{proof}
By the Keller theorem, the convex compact space $P[0;1]$ can be affinely 
embedded into the Hilbert space, hence into every
Banach space with a
Schauder basis (hence into every Banach space). Hence
$(1)$ implies $(2)$.

In order to check $(1),$
pick an arbitrary absolutely continuous measure $\mu \in K=P[0;1]$, for example the
Lebesgue measure. For every $m \in K\setminus \{\mu\},$ denote by $l_{m,\mu}$ the infinite ray from the point $m$ through the point $\mu$. Define the {\it convex complement}
of $\mu$ by setting
$$
Y =\{m \in K\setminus \{\mu\}\,\,: \quad l_{m,\mu} \cap K = [m;\mu] \}.
$$
Clearly $Y$ is convex. For every point $x \in [0;1],$ the Dirac measure\index{Dirac measure} $\delta_x$ belongs to $Y$ because  $(1-t)\delta_x  + t\mu,  t>1$ is not a probability measure\index{probability measure}. 
Hence $extr(K) \subset Y$.
Next,
$$
Y = \bigcap_{n=1}^{\infty} \{m \in K\,\,: (1-n^{-1}) \cdot \delta_x  + n^{-1} \cdot \mu \notin K \}
$$
and this is why $Y$ is a proper convex $G_{\delta}$-subset of $K$.

Finally, define $R:K \to K$ by setting $R(m) = \{m'\in K: supp(m') \subset supp(m)\},$ 
where $supp$ denotes the {\it support} of the
probability measure, i.e. the set of all points $x \in [0;1]$ with 
the property that the value of the
measure is positive over each neighborhood of the point. It is a straightforward verification that $R:K \to K$ is a LSC convexvalued mapping and
that the equality
$$
R\left( \sum \lambda_i \cdot \delta_{x_i} \right)  = conv\{\delta_{x_1},...,\delta_{x_n}\},\quad x_i \in [0;1], \quad \lambda_i \geq 0, \sum \lambda_i =1,
$$
is evidently  true.
\qed
\end{proof}

A version of the construction was proposed in \cite{RSClos} which (formally) avoids any probability measures and works directly in the Hilbert cube $Q=[0;1]^{\mathds{N}}$. Here is a sketch:

\begin{itemize}
\item{} $X = \{x \in Q : x_{1}=1,\,
x_{n}=x_{2n}+x_{2n+1}, n \in \mathds{N}\}$;
\item{} $Y = \{ x \in
X : \sup\{ x_{n}z_{n}^{-1}: n \in \mathds{N}\} = \infty \}$, where $z
\in X$ is arbitrarily chosen so that $\lim_{n \to \infty} z(n) =
0$ and $z(n)> 0$ for all $n$; and
\item{} $F(x)=\Phi(x)\cap
Y$, where $\Phi: X \rightarrow X$ is defined by
$$
\Phi(x) = \{y \in X : y_{n}=0 \ \  \hbox{whenever} \ \  x_{n} = 0 \}.
$$
\end{itemize}

Let us temporarily say
that natural numbers $2n$ and $2n+1$ are {\it sons} of the number $n$,
which in turn, we shall call the {\it father} of such {\it twins}. Thus each
natural number has exactly 2 sons, 4 grandsons, etc. and the
natural partial order, say $\prec$,
immediately arises on the set $\mathds{N}$. 
With respect to $\prec$, the set $\mathds{N}$ can be
represented as a
binary tree $T$ and every $x \in X$ is a mapping $x : T \rightarrow [0;1]$ with
$x_{1}=1,\, x_{n}=x_{2n}+x_{2n+1}, n \in \mathds{N}$. In other words, each $x \in X$
defines some probability distribution\index{probability distribution} 
on each $n$-th level of the binary tree $T$.

Hence even though
all proofs in this construction can be performed directly in the Hilbert cube, the set $X$ is in fact,
a "visualization" of the set of all probability measures of the Cantor set and  details of the proof look similar to those above.

In conclusion, we mention the paper \cite{FilGe} which demonstrated the essentiality of the $G_{\delta}$-assumption for $Y$ in Theorem 5 of Michael and Namioka. Briefly, it was proved that for every countable $A \subset [0;1],$ the set $Y=P_A = \{\mu \in P[0;1] :  supp(\mu) \subset A\}$ has the property $(1)$ from Theorem 5.
Then by using sets of probability measures with various countable supports,
the
authors
constructed a convex subset $Y \subset \R^{2} \times l_2$ with property $(1)$ and without property $(2)$ from Theorem 5.
Note that, as it was proved by V.
Kadets, the property $(1)$ from Theorem 5 is equivalent to the
closedness of
$Y \subset B$ being not only a convex set,  but also
a {\it linear subspace} (cf.  \cite[Proposition 5.1]{M10}).

\section{Selections and extensions}
\label{sec:3}

There are intimate relations between selections and 
extensions\index{extension}
and typically they appear together:  if $A \subset X$ and $f:A \to Y$ then $\widehat{f}: X \to Y$ is an extension of $f$ if and only if $\widehat{f}$ is a selection of multivalued mapping $F_A: X \to Y$ defined by setting $F_A(x)=\{f(x)\}, x \in A,$ and $F_A(x)=Y$ otherwise. 
Thus
as a rule,
each fact concerning existence of singlevalued selections implies some result on
extensions. In the other direction,
many basic theorems (or some of their special cases) about
extensions are special cases of some appropriate
selection theorems.

However, extension theory is certainly not simply a subtheory of selection theory: specific questions and problems need specific ideas and methods. For example, selection Theorem 1 implies that every continuous map $f:A \to Y$ from a closed subset $A$ of a paracompact domain $X$ into a Banach space $B$ has a continuous extension $\widehat{f}: X \to B$ with $f(X) \subset Clos(conv(f(A)))$. However, the Dugundji extension theorem states somewhat differently: every continuous map $f:A \to Y$ from a closed subset $A$ of a metric (or stratifiable) domain $X$ into a locally convex topological vector space $B$ has a continuous extension $\widehat{f}: X \to B$ with $f(X) \subset conv(f(A))$. Besides the differences in assumptions and conclusions,
these two ``similar'' theorems are proved by almost disjoint techniques: a sequential procedure of some approximations in the Michael selection theorem and a straightforward answer by a
formula in Dugundji extension theorem. It seems that the only common point are continuous partitions of unity\index{partition of unity}.

To emphasize the difference on a more nontrivial level, let us recall that for a wide class of nonlocally convex, completely metrizable, topological vector spaces it was proved in
\cite{D} that all such spaces are 
{\it absolute retracts}\index{absolute retract}
(with respect to all metrizable spaces), abbreviated as $AR's$. At the same time, at present
there is no known
example of
a
nonlocally convex, completely metrizable, topological vector space $E$ which can be successfully substituted instead of Banach (or Fr\'{e}chet) spaces $B$ into the assumption of the Michael selection Theorem 1. In particular, Dobrowolski stated (private communication) the following:

\begin{problem} \label{TDobr}
Is the space $l_p,   0<p<1$, of all $p$-summable sequences of reals an absolute selector, i.e. is it true that for every paracompact space (metric space, compact space) and
every LSC mapping $F:X \to l_p$ with convex closed values,
there exists a continuous singlevalued selection of $F$?
\end{problem}

During the last decade one of the most interesting facts concerning
relations between selections and extensions was obtained by Dobrowolski and van Mill \cite{DM}.
To explain their main results recall that $g:X \to Y$ is said to be an 
{\it $\varepsilon$-}selection\index{$\varepsilon$-selection} 
of a multivalued mapping $F:X \to Y$ into a metric space $(Y;d)$ if $dist(g(x), F(x))< \varepsilon$. Dobrowolski and van Mill used the term 
{\it $\epsilon$-near selection}\index{$\epsilon$-near selection}
for the case when the
strong inequality $dist(g(x), F(x))< \varepsilon$ is replaced
by $dist(g(x), F(x)) \leq \varepsilon$. Clearly, for closedvalued mappings $0$-near selections are exactly selections.

\begin{definition} \label{FDimSelPr}
A convex subset $Y$ of a vector metric space $(E;d)$ has the 
{\it finite-dimensional selection property}\index{finite-dimensional selection property}
(resp. 
{\it finite-dimensional near selection property}\index{finite-dimensional near selection property})
if for every metrizable domain $X$ and 
every LSC mapping $F:X \to Y$ with all compact convex  and finite-dimensional values $F(x) \subset Y,\,\,x \in X$, there exists a continuous singlevalued selection of $F$ (resp.,
for every $\varepsilon > 0$ there exists a continuous singlevalued $\varepsilon$-near selection of $F$).
\end{definition}

Combining 3.3, 4.1, 5.4 and 6.1 from \cite{DM} we formulate the following:

\begin{theorem} [\cite{DM}] \label{ARSel}
For any convex subset $Y$ of a vector metric (not necessarily, locally convex) space $(E;d)$ the following statements  are equivalent:
\begin{description}
\item{(1)} $Y$ is an $AR$; 
\item{(2)} $Y$ has the finite-dimensional near selection property.
\end{description}
\end{theorem}

As for a specific selection theorem we cite:

\begin{theorem} [\cite{DM}] \label{5.4}
Let  $Y$ be a convex subset of a vector metric (not necessarily locally convex) space $(E;d)$. Then for every metrizable domain $X$ and
every compactvalued and convexvalued LSC mapping $F:X \to Y$ with $\max \{dim F(x):\,\,x \in X\} < \infty$, there exists a continuous singlevalued selection $f$ of $F$.
\end{theorem}

Continuous singlevalued selections $f$ of a given multivalued mapping $F$ are usually constructed as uniform limits of  sequences of
certain approximations $\{ f_{\varepsilon} \}, \varepsilon \to 0,$ of $F$. Practically all known
selection results are obtained by using one of the following two approaches
for a construction of $\{ f_n = f_{\varepsilon_n} \} , \,\, \varepsilon_{n} \to 0$. In the first (and the most popular) one,
the method of outside approximations\index{method of outside approximations}, 
mappings $f_n$ are
continuous $\varepsilon_{n}$-selections of $F$, i.e. $f_{n}(x)$ all lie {\it near}
the set $F(x)$ and all mappings $f_n$ are {\it continuous}. 
In the second one,
the method of inside approximations\index{method of inside approximations},
$f_n$ are $\delta_n$-continuous selections of $F$, i.e. $f_n(x)$  all lie {\it in} the set $F(x)$, however $f_n$ are {\it discontinuous}.

We emphasize that Theorem 9 is proved by using the method of inside approximation.
This
is a rare situation: all previously known to us examples are \cite{BP, MConv, vMB,  RS}. 
However,
for nonlocally convex range spaces it is an adequate approach since
for such  spaces $Y$ the intersections of  convex subsets with balls are in general, nonconvex. Also, compactness is not preserved under the convex closed hull operation.

It is very natural to try to substitute $dim F(x) < \infty,\,\,x \in X,$ in Theorem 9
instead of $\max \{dim F(x):\,\,x \in X\} < \infty$. 
It turns out that this is a futile attempt. Namely, 5.6 in \cite{DM} implies that Theorem 9  becomes false with such a change of the
dimensional restriction.

\begin{theorem} [\cite{DM}]\label{ExDobMi}
There exist a linear metric vector space $E$ and a LSC mapping $F: Q \to E$ from the Hilbert cube $Q,$ such that $E$ contains the tower $\{E_n\}$ of closed subsets with the following properties:
\begin{description}
\item{(1)} $Q= \cup Q_n,\, where \,\,Q_n = F^{-1}(E_n)$;
\item{(2)} The restrictions $F|_{Q_n}$ satisfy all assumptions of Theorem 9; and
\item{(3)} For arbitrary choices of continuous selections $f_n:Q_n \to E$ of $F|_{Q_n}$ their pointwise limit $f=\lim_{n}f_n$ is not a continuous mapping, whenever such a pointwise limit exists.
\end{description}
\end{theorem}

Note that due to the Localization principle (Lemma 1) the assumption $\max \{dim F(x):\,\,x \in X\} < \infty $ in Theorem 9 can be replaced by its local version $\max \{dim F(x):\,\,x \in U(x) \} < \infty $ for some neighborhood $U(x)$ of $x$. Two slight generalizations of Theorem 9 were presented in \cite{RSClos}: in the first one $Y$ was replaced by a $G_{\delta}$-subset and in the second the closedness restriction for values $F(x)$ was omitted.

\begin{theorem}[\cite{RSClos}] \label{GenSelDobMi}
\begin{description}
\item{(1)}
Let $F: X \rightarrow Y$ be
a LSC convexvalued mapping of a paracompact domain $X$ into a
$G_{\delta}$-subset $Y$ of a completely metrizable linear space $E$.
Then $F$ admits a singlevalued continuous selection provided that
the values $F(x)$ are closed in $Y$ and that for every $x \in X$
there exists a neighborhood $U(x)$ such that $max\{dim F(x'): x'
\in U(x)\}< \infty$.
\item{(2)} Let $F: X \rightarrow E$ be a
LSC convexvalued mapping of a metrizable domain $X$ into a
completely metrizable linear space $E$. Then $F$ admits a
singlevalued continuous selection provided that for every $x \in
X$ there exists a neighborhood $U(x)$ such that $max\{dim F(x'):
x' \in U(x)\}< \infty$.
\end{description}
\end{theorem}

It is interesting to note that the metrizability of the
domain in Theorem 11$(2)$ (in comparison with the paracompactness in $(1)$) is an essential restriction because the proof is based on the
density selection theorem of Michael \cite{M58} which works exactly for metrizable spaces.

Under dimensional restrictions for the
domain, not for the values of the mapping, van Mill  \cite[Cor.5.2]{vMB} obtained the following:

\begin{theorem}[\cite{vMB}]
Let $X$ be a locally finite-dimensional paracompact space and $Y$ a  convex subset of a
vector metric space. Then each LSC mapping $F: X \to Y$ with complete convex values admits a singlevalued continuous selection.
\end{theorem}

On the other hand, Example 5.3 \cite{DM} shows that in general,
Theorem 12 does not  hold for domains which are unions of countably many finite-dimensional compacta.
However,
if in the assumptions of Theorem 12 one passes to $C$-domains (which look as approximately finite-dimensional spaces) then exact selections can be replaced by $\varepsilon$-selections \cite[Theorem 6.3]{DM}.

\begin{theorem} [\cite{DM}]
Let $X$ be a $C$-space and  $Y$
a  convex subset of a
vector metric space. Then each LSC mapping $F: X \to Y$ with convex values admits a singlevalued continuous $\varepsilon$-selections for any $\varepsilon >0$.
\end{theorem}

In a volumninous paper, Gutev, Ohta and Yamazaki \cite{GOY-2} systematically used selections and extensions for obtaining the criteria for various kinds of displacement of a subset in the entire space.
Recall that $A \subset X$ is 
$C$-embedded\index{$C$-embedding} 
in $X$ (resp., 
$C^{*}$-embedded\index{$C^{*}$-embedding}
in $X$) if every continuous (resp., every bounded continuous) function $f:A \to \mathds{R}$ has a continuous extension to entire $X$. Below are some of their typical results, Theorems 4.3, 4.6, and 6.1.

\begin{theorem} [\cite{GOY-2}]\label{C-starEmb}
For a subset $A$ of $X$ the following statements  are equivalent:
\begin{description}
\item{(1)} $A$ is $C^{*}$-embedded in $X$;
\item{(2)}  For every Banach space $B$, every continuous mapping $F:X \to B$ with compact convex values $F(x)$ and every continuous selection $g: A \to B$ of the restriction $F|_A$, there is a continuous extension $f: X \to B$ of $g$ which is also a selection of $F$;
\item{(3)} The same as $(2)$, but without convexity of $F(x)$ and without $f$ being a selection of $F$; 
\item{(4)} For every cardinal $\lambda$, every continuous maps $g,h: X \to c_{0}(\lambda)$ with $g(a)\leq h(a),  a \in A,$ and every  $f: A \to c_{0}(\lambda)$ which separates $g|_A$ and $h|_A$,
there exists a continuous extension $\widehat{f}:X \to c_{0}(\lambda)$ of $f$.
\end{description}
\end{theorem}

Here $c_{0}(\lambda)$ denotes the Banach space of all mappings $x$ from $\lambda$ to the reals such that the sets $\{\tau < \lambda:\,\,|x(\tau)| \geq \varepsilon\}$ are finite for all $\varepsilon>0$. Note that $c_{0}(1)=\mathds{R}$ and $c_{0}(\aleph_0)=c_0$.

\begin{theorem} [\cite{GOY-2}] \label{C-Emb}
For a subset $A$ of $X$ the following statements  are equivalent:
\begin{description}
\item{(1)} $A$ is $C$-embedded in $X$;
\item{(2)}  For every Banach space $B$,  every lower $\sigma$-continuous mapping $F:X \to B$ with closed convex values $F(x)$ and  every continuous selection $g: A \to B$ of the restriction $F|_A$, there is a continuous extension $f: X \to B$ of $g$ which is also the selection of $F$;
\item{(3)} The same as $(2)$, but without convexity of $F(x)$ and without $f$ being a selection of $F$; 
\item{(4)} For every cardinal $\lambda$,  every continuous $g_n,h_n: X \to c_{0}(\lambda), n \in \mathds{N},$ with $\liminf g_n(a) \leq \limsup h_n(a),  a \in A,$ and every  $f: A \to c_{0}(\lambda)$ which separates $\liminf g_n(a)$ and $\limsup h_n(a)$, there exists a continuous extension $\widehat{f}:X \to c_{0}(\lambda)$ of $f$.
\end{description}
\end{theorem}

Here {\it 
lower $\sigma$-continuity}\index{lower $\sigma$-continuous}
of a multivalued map means that it is the pointwise closure of a
union of countably many continuous compactvalued mappings.
Yet another characterization of $C$-embeddability can be formulated via mappings into open convex subsets of a Banach spaces.

\begin{theorem} [\cite{GOY-2}] \label{CC-Emb}
For a subset $A$ of $X$ the following statements  are equivalent:
\begin{description}
\item{(1)} $A$ is $C$-embedded in $X$;
\item{(2)}  For every Banach space $B$, every open convex $Y \subset B$, every lower $\sigma$-continuous mapping $F:X \to Clos(Y)$ with closed convex values $F(x)$ and  every continuous selection $g: A \to B$ of $F|_A$ with $g^{-1}(Y)=A \cap F^{-1}(Y)$, there is a continuous extension $f: X \to B$ of $g$ which is also the selection of $F$ and $f^{-1}(Y)=F^{-1}(Y)$; 
\item{(3)} The same as $(2)$, but with continuous compactvalued mapping $F$.
\end{description}
\end{theorem}

One more recent paper in which selections and extensions are simultaneously studied is  due to
Michael \cite{M11}. Below we unify 3.1 and 4.1 from \cite{M11}.

\begin{theorem} [\cite{M11}]\label{M11}
For a metrizable space $Y$  the following statements  are equivalent:
\begin{description}
\item{(1)} $Y$ is completely metrizable;
\item{(2)} For every paracompact domain $X$ and every LSC mapping $F:X \to Y$ with closed values, there exists a LSC selection $G:X \to Y$ with compact values;
\item{(3)} For every closed subset $A$ of a paracompact domain $X$ and  every continuous $g: A \to Y$, there exists a LSC mapping $G:X \to Y$ with compact values which extends $g$;
\item{(4)} Similar to $(2)$ but for USC selection $H:X \to Y$ with compact values;
\item{(5)} Similar to $(3)$ but for USC extension $H:X \to Y$ with compact values.
\end{description}
\end{theorem}

The implication $(5) \Rightarrow (1)$ is true in a more general case, namely when $Y$ is Czech-complete and $X$ is a paracompact $p$-space \cite{NV_Ext,Psz}.

To finish the section we return once again to comparison of the Dugundji extension theorem and the Michael selection theorem. Arvanitakis \cite{Arv} proposed a uniform approach to proving both of these theorems. He worked with paracompact $k$-domains. Recall that a Hausdorff space $X$ is called a 
$k$-{\it space}\index{$k$-space}
if closedness of $A \subset X$ coincides with closedness of $A \cap K$ for all compact $K \subset X$. 
Below,
completness of a locally convex vector space $E$ means that the closed convex hull operation preserves the compactness of subsets in $E$. Next, $C(T;E)$ denotes the vector space
of all continuous mappings from a topological space $T$ into $E$,
endowed by the topology of uniform convergence on compact subsets.

\begin{theorem} [\cite{Arv}]\label{Arvani}
Let $X$ be a paracompact $k$-space, $Y$ a complete metric space, $E$ a locally convex complete vector space, and $F:X \to Y$ a LSC mapping. Then there exists a linear continuous operator $S:C(Y;E) \to C(X;E)$ such that
$$ S(f)(x) \in Clos(conv(f(F(x)))), \quad f \in C(Y;E), \quad x \in X.$$
\end{theorem}

The proof is based on study of regular 
{\it exaves} (extensions/averagings) operators but without any explicit use of 
probability measures and 
Milyutin mappings\index{Milyutin mapping}.
Recently, Valov \cite{VArv} suggested a generalization of this
theorem to the case of an arbitrary paracompact domain. He extensively exploited in its full
force the technique of functors $P_\beta$ and $\widehat{P}$, 
averaging operators\index{averaging operator},
Milyutin mappings, and so on. He also used the universality of the zero-dimensional selection Theorem 2, i.e. the fact obtained in \cite{RSS} that Theorem 2 implies both Theorems 1 and 3.

\begin{theorem}[\cite{VArv}] \label{ValArvani}
Let $X$ be a paracompact space, $Y$ a complete metric space and $F:X \to Y$ a LSC mapping. Then:
\begin{description}
\item{(1)}  For every locally convex complete vector space $E$ there exists a linear  operator $S_b:C_{b}(Y;E) \to C_{b}(X;E)$ such that
$$ S(f)(x) \in Clos(conv(f(F(x)))), \quad f \in C(Y;E), \quad x \in X,$$
and such that $S_b$ is continuous with respect to the uniform topology and the topology
of uniform convergence on compact subsets;
\item{(2)} If $X$ is a $k$-space
and
$E$ is a Banach space then $S_b$ can be continuously extended (with respect to both topologies) to a linear  operator $S:C(Y;E) \to C(X;E)$ with the  property  that $ S(f)(x) \in Clos(conv(f(F(x))))$.
\end{description}
\end{theorem}

Therefore
by taking $Y=E=B$ to be a Banach space, $F$ a mapping with closed convex values and $f=id|_Y,$
one can see that $S(f)$ is a selection of $F$: $S(id)(x) \in Clos(conv(F(x))),\,\, x\in X.$

Next, if $A$ is a
completely metrizable closed subspace of $X$, $E$ a locally convex complete vector space,
and $F=F_A$ a mapping defined by $F(x)=\{x\}, x\in A,$
and $F(x)=A,\, x \in X \setminus A,$ then we see  that $S_{b}(f)(x) \in Clos(conv(f(F(x)))),$
for any $f \in C_{b}(A;E)$ and hence $S_{b}(f)(x)=f(x)$, whenever $x \in A$. Therefore
$S_{b}(f)$ is an extension of $f$.
Thus the result is on the one hand stronger than the Dugundji theorem because $X$ can be nonmetrizable, but on the other hand it is weaker because $A$ should be completely metrizable and the result relates to $C_{b}(A;E)$,
not to $C(A;E)$.

As a corollary, the Banach-valued version of the celebrated
Milyutin theorem can be obtained:

\begin{theorem} [\cite{VArv}] \label{ArvanMil}
Let $X$ be an uncountable compact metric space, $K$ the Cantor set and $B$ a Banach space. Then $C(X;B)$ is isomorphic to $C(K;B)$.
\end{theorem}

\section{Selection characterizations of domains}
\label{sec:4}

Theorem 1 states that assumptions $(1)-(5)$ imply the existence of selections of a multivalued mapping $F$. Conversely, assumptions $(2)-(5)$ together with existence of selections imply the condition $(1)$ that a domain $X$ is a paracompact space. In other words,
Theorem 1 gives a {\it selection characterization of paracompactness}. By varying the types of the range Banach spaces $B$, types of families of convex subsets of $B$, types of continuity of $F$, etc. one can try to find a selection characterization of some other topological types of domains. Originally, Michael \cite{Sel-1} found such types of characterization for normality, collectionwise normality, normality and countable paracompactness, and
perfect normality. Below we concentrate on recent results in this direction.

Gutev, Ohta and Yamazaki \cite{GOY-1} obtained selection characterizations for three classes of domains inside the class of all 
$\lambda$-collectionwise normal spaces\index{$\lambda$-collectionwise normal space}.
Recall that this property means that for each discrete family $\{F_{\gamma}\}_{\gamma \in \Gamma}$ of
closed subsets with $\left|\Gamma \right| \leq \lambda$ there is a discrete family $\{G_{\gamma}\}_{\gamma \in \Gamma}$ of open sets such that $F_{\gamma} \subset G_{\gamma}$.
Note that the equivalence of $(1)$ and $(2)$ in Theorem 21
was proposed by Michael \cite{Sel-1} 
(cf. the discussion concerning the proofs in Ch. II of \cite{RS}).
We also observe that $(4)$ in  Theorems 21-13 resembles the classical Dowker separation theorem.

\begin{theorem}[\cite{GOY-1}]\label{GOY_1}
Let $\lambda$ be an infinite cardinal. Then for any $T_1$-space $X$ the following statements  are equivalent:
\begin{description}
\item{(1)} $X$ is $\lambda-$collectionwise normal;
\item{(2)} For every Banach space $B$ of the weight less than or equal to $\lambda$
and every LSC mapping 
$F: X \to B$ whose values $F(x)$ are convex compacta, or $F(x)=B$, there exists a continuous singlevalued selection of $F$;
\item{(3)}  Same as $(2)$ but for the Banach space $B=c_{0}(\lambda)$;
\item{(4)}  For every closed $A \subset X$ and  every singlevalued $g,h : A \to c_{0}(\lambda)$ such that $g \leq h$, $g$ is upper semicontinuous, and
$h$ is lower semicontinuous, there exists a singlevalued continuous $f : X \to c_{0}(\lambda)$ such that $f|_A$ separates $g$ and $h$, i.e. $g \leq f|_A \leq h$.
\end{description}
\end{theorem}

\begin{theorem} [\cite{GOY-1}] \label{GOY_2}
Let $\lambda$ be an infinite cardinal. Then for any $T_1$-space $X$ the following statements  are equivalent:
\begin{description}
\item{(1)} $X$ is countably paracompact and $\lambda$-collectionwise normal;
\item{(2)} For every generalized $c_{0}(\lambda)$-space $B$ and every LSC mapping
$F: X \to B$ with values $F(x)$ being convex compacta, or $F(x)=B$, and with $|F(x)|>1, x \in X$, there exists a continuous singlevalued selection $f$ of $F$ such that $f(x)$ is not an extreme point of $F(x), \,x\in X$;
\item{(3)}  Same as $(2)$ but for the Banach space $B=c_{0}(\lambda)$;
\item{(4)}  Same as $(4)$ in Theorem 21 but with strong inequalities $g < h$ and $g < f|_A < h$.
\end{description}
\end{theorem}

\begin{theorem} [\cite{GOY-1}]\label{GOY_3}
Let $\lambda$ be an infinite cardinal. Then for any $T_1$-space $X$ the following statements  are equivalent:
\begin{description}
\item{(1)} $X$ is perfectly normal and $\lambda$-collectionwise normal;
\item{(2)} For every generalized $c_{0}(\lambda)$-space $B$ and  every LSC mapping $F: X \to B$ with values $F(x)$ being convex compacta, or $F(x)=B$,  there exists a continuous singlevalued selection $f$ of $F$ such that $f(x)$ is not an extreme point of $F(x)$, whenever  $|F(x)|>1$;
\item{(3)}  Same as $(2)$ but for the Banach space $B=c_{0}(\lambda)$;
\item{(4)}   Same as $(4)$ in Theorem 21 but with strong inequalities  $g(x) < f(x) < h(x)$ for all $x \in A$ with $g(x)< h(x)$.
\end{description}
\end{theorem}

One of the key ingredients of the proofs is the fact that for a closedvalued and convexvalued mapping $F,$ a selection avoiding all extreme points exists provided that $F$ admits two families of local disjoint selections. Certainly Theorems 21-23 constitute a base for Theorems 14-16 above.

In \cite{GOY-1} authors stated the following question: Do
Theorems 22 and 23 remain valid
if in $(2)$ one replaces $c_{0}(\lambda)$-space by an arbitrary Banach space $B$ of weight less than or equal to $\lambda$? Yamauchi  answered this question 
in the affirmative.

\begin{theorem} [\cite{Yextr}] \label{Yamauch_1}
Let $\lambda$ be an infinite cardinal. Then for any $T_1$-space $X$ the following statements  are equivalent:
\begin{description}
\item{(1)} $X$ is countably paracompact and $\lambda$-collectionwise normal;
\item{(2)} For every Banach space $B$ of 
weight less than or equal to $\lambda$,  every LSC mapping $F: X \to B$ with values $F(x)$ being convex compacta, or $F(x)=B$  and with $|F(x)|>1, x \in X$,
there exists a continuous singlevalued selection $f$ of $F$ such that $f(x)$ is not an extreme point of $F(x), \,x\in X$.
\end{description}
\end{theorem}

Passing to $\lambda$-paracompactness,
the following was proved in \cite[Theorem 8]{Yextr}:
\begin{theorem} \label{Yamauch_2}
Let $\lambda$ be an infinite cardinal. Then
for any $T_1$-space $X$ the following statements
are equivalent:
\begin{description}
\item{(1)} $X$ is normal and $\lambda$-paracompact; 
\item{(2)} The same as $(2)$ in Theorem 24 but with closed values $F(x)$.
\end{description}
\end{theorem}

Considering $\lambda=\aleph_0$ one observes that in the Michael selection criteria for $X$ being  normal and countably paracompact one can assume that a selection $f$ always avoids extreme points of values of multivalued mapping $F$ with $|F(x)|>1, x\in X$.

Analogously, a domain $X$ is perfectly normal and $\lambda$-paracompact if and only if for every Banach space $B$ with $w(B) \leq \lambda$ and every LSC mapping $F: X \to B$ with convex closed values $F(x)$ (not necessarily with $|F(x)|>1, x\in X$\,) there exists a continuous singlevalued selection $f$ of $F$ such that $f(x)$ is not an extreme point of $F(x)$, whenever  $|F(x)|>1$ (cf.  \cite{Yextr}).

Before stating one more result recall that {\it normality of a covering $\omega$} means an existence of a sequence $\omega_1=\omega,  \omega_2, \omega_3,...$ of coverings such that each $\omega_{n+1}$ is
a strong star refinement of $\omega_n$ and that a space is called a
$\lambda-PF$-{\it normal space}\index{$\lambda-PF$-normal space}
if each its point-finite open coverings is normal. Yamauchi \cite{YSim}
characterized the class of $\lambda-PF$-normal  spaces.

\begin{theorem}[\cite{YSim}]\label{Yamauch_3}
Let $\lambda$ be an infinite cardinal. Then for any $T_1$-space $X$ the following statements  are equivalent:
\begin{description}
\item{(1)} $X$ is $\lambda-PF$-normal;
\item{(2)} For every simplicial complex $K$ with cardinality less than or equal to $\lambda$,  every simplex-valued LSC mapping $F: X \to |K|$  has
a continuous singlevalued selection.
\end{description}
\end{theorem}

An analogue of Theorem 26 for dimensional-like properties was also given in \cite{YSim}.

Here $|K|$ stands for a 
{\it geometric realization}\index{geometric realization}
of $K$, for example in the Banach space $l_{1}(Vert_K)$, where $Vert_K$ is the set of vertices
of $K$, and  $|K|$ is endowed with the metric topology, induced by this embedding. In fact, the initial result here was a theorem of Ivan\v{s}i\'{c} and Rubin \cite{IR} that every simplex-valued mapping $F: X \to |K|_w$ admits a selection provided that  $F: X \to |K|_w$ is locally selectionable, where $|K|_w$ denotes $|K|$ endowed with the weak topology.

Yamauchi \cite{YBlum} proposed selection criteria for classes of 
realcompact spaces\index{realcompact space}, 
Dieudonn\'{e} complete spaces\index{Dieudonn\'{e} complete space}
and 
Lindel\"{o}f spaces.
The starting point was the result of  Blum and Swaminatham \cite{BlSw} on selection characterization of realcompactness in terms of the so-called 
$\mathcal{S}$-fixed LSC mapping\index{$\mathcal{S}$-fixed LSC mapping} 
into a locally convex topological vector space. 
To avoid specific notations we collect here only the results for Lindel\"{o}f spaces.

\begin{theorem} [\cite{YBlum}]\label{Yamauch_4}
For any regular space $X$ the following statements  are equivalent:
\begin{description}
\item{(1)} $X$ is Lindel\"{o}f;
\item{(2)} For every completely metrizable space $Y$ and  
every closedvalued LSC mapping $F:X \to Y$ there exist
compactvalued USC mapping $H:X \to Y$ and compactvalued LSC mapping $G:X \to Y$ such that  $G(x) \subset H(x) \subset F(x), x \in X$ and $H(X)=\bigcup\{H(x): x\in X\}$ is separable;
\item{(3)} For every Banach space $B$ and every LSC mapping $F:X \to B$ with closed convex values there exists a continuous singlevalued selection $f$ of $F$ with separable image $f(X)$.
\end{description}
\end{theorem}

Next, we refer to our previous survey on selections to cite the Choban-Gutev-Nedev conjecture.

\begin{problem} [\cite{RS-2}]\label{ChGuNe}
For every $T_1$ space $X$ the following statements  are equivalent:
\begin{description}
\item{(1)} $X$ is countably paracompact and collectionwise normal;
\item{(2)} For every Hilbert space $H$ and every LSC mapping $F:X \to H$ with closed convex values there exists a continuous singlevalued selection $f$ of $F$.
\end{description}
\end{problem}

The implication $(2) \Rightarrow (1)$ is a standard exercise, while $(1) \Rightarrow (2)$ was a hard problem. During the last decade, in a series of papers, Shishkov  successfully resolved the problem step by step. Here is a short list of his results.
First, he reduced the situation to the case of 
{\it bounded} mappings\index{bounded mapping}
$F: X \to Y$, i.e. mappings with the bounded in $Y$ image $F(X)=\bigcup \{F(x): x \in X \}$.

\begin{theorem} [\cite{Sh2000}]\label{Sh_1}
For every countably paracompact space $X$ and every normed space $Y$ the following statements  are equivalent:
\begin{description}
\item{(1)} For every LSC mapping from $X$ to $Y$ with closed convex values there exists a continuous singlevalued selection;
\item{(2)} Same as $(1)$ but for bounded LSC mappings.
\end{description}
\end{theorem}

Next, he solved the problem in the case of the domain $X$ 
a $\sigma$-product of metric spaces \cite{Sh2001} and extended LSC mappings with normal and countably paracompact domains over the Dieudonn\'{e} completions of the domains \cite{Sh2002}. Then he proved the conjecture for domains which are hereditarily ``nice'' \cite{Sh2004}:

\begin{theorem} [\cite{Sh2004}] \label{Sh_2}
Let $X$ be a countably paracompact and hereditarily collectionwise normal space, $B$ a reflexive Banach space and $F: X \to B$ a LSC mapping with convex closed values. Then there exists a continuous singlevalued selection $f$ of $F$.
\end{theorem}

In  \cite{Sh2007} Shishkov worked with a paracompactness-like restriction on domain.

\begin{theorem} [\cite{Sh2007}] \label{Sh_3}
Let $X$ be a $\mathfrak{c}$-paracompact and collectionwise normal space, $B$ a reflexive Banach space, and $F: X \to B$ a LSC mapping with convex closed values. Then there exists a continuous singlevalued selection $f$ of $F$.
\end{theorem}

Note that $X$ from the last theorem can be a nonparacompact space, and that collectionwise normality plus $\mathfrak{(c)}$-paracompactness of domain is not in general, a necessary restriction for existence od selections
(cf.  Nedev's theorem for mappings over $\omega_1$ \cite{N}).

Then Shishkov \cite{Sh2008} proved the following fact which,
together with Theorem 28,
finally resolved Choban-Gutev-Nedev problem.

\begin{theorem} [\cite{Sh2008}] \label{Sh_4}
Let $X$ be a collectionwise normal space, $H$ a Hilbert space and $F: X \to H$ a LSC mapping with convex closed and bounded values. Then there exists a continuous singlevalued selection $f$ of $F$.
\end{theorem}

We note that the proof essentially uses the geometric and analytical structure of a Hilbert space. Thus for reflexive range spaces the problem is still open.

We end the section by recent results of Gutev and Makala \cite{GM-1} who have
suggested a characterization for classes of domains by using a controlled local improvement of $\varepsilon$-selections up to a genuine selection, rather than using exact selections.

\begin{theorem} [\cite{GM-1}] \label{GuMak_1}
Let $\lambda$ be an infinite cardinal. For any $T_1$-space $X$ the following statements  are equivalent:
\begin{description}
\item{(1)} $X$ is countably paracompact and $\lambda$-collectionwise normal;
\item{(2)} For every Banach space $B$ with $w(B) \leq \lambda$, every LSC mapping $F: X \to B$ with values $F(x)$ being convex compacta, or $F(x)=B$, every continuous $\varepsilon: X \to (0;+\infty)$ and every $\varepsilon$-selection $f_{\varepsilon}:X \to B$ of $F$, there exists a continuous singlevalued selection $f$ of $F$ such that $dist(f(x),f_{\varepsilon}(x) < \varepsilon(x), \,\,x \in X$.
\end{description}
\end{theorem}

Similarly  we have for normality:
\begin{theorem}[\cite{GM-1}]\label{GuMak_2}
For any $T_1$-space $X$ the following statements  are equivalent:
\begin{description}
\item{(1)} $X$ is countably paracompact and normal;
\item{(2)} Same as $(2)$ in Theorem 32 but for separable Banach spaces.
\end{description}
\end{theorem}

Note that the starting point of proofs in \cite{GM-1} was the following:

\begin{theorem} [\cite{GM-1}] \label{GuMak_3}
For any  Banach space $B$ the following statements  are equivalent:
\begin{description}
\item{(1)} For every collectionwise normal domain $X$ and every LSC $F: X \to B$ with values $F(x)$ being convex compacta, or $F(x)=B$ there exists a continuous selection $f$ of $F$;
\item{(2)} Same as $(1)$ but without possibility of $F(x)=B$.
\end{description}
\end{theorem}

\section{Generalized convexities}
\label{sec:5}

{\bf 5.1.}\,Roughly speaking, there exists an entire mathematical "universe" devoted to
various generalizations and versions of convexity. In our
opinion, even if one simply lists the titles of  "generalized
convexities" one will find as a minimum, nearly 20 different
notions.

As for the specific situation with continuous selections perhaps
two principal
approaches are really used here. With the {\it inner} point of view, one starts by introducing
some type of ``convex hull'' operation and defines a convex set as a set which is preserved
by such an operation. Typical examples are 
Menger's metric convexity \cite{Men}, 
Michael's convex and geodesic structures \cite{MConv},
M\"{a}gerl's paved spaces \cite{Mag}, 
Bielawski's simplicial convexity \cite{Bel},
Horvath's structures \cite{Hor}, 
Saveliev's convexity \cite{Sav},
etc.

With respect to {\it outer} constructions, convex sets are introduced by some list of axioms and then the convex hull $conv A$ of a set $A$ is defined as the intersection $\cap\{C:  A \subset C,\,\,C \,\,{\rm is \,convex}\}$. Among examples are: 
Levy's abstract convexity \cite{Lev},
Jamison's convexity \cite{Jam}, 
van de Vel's topological convexity \cite{Vel}, 
decomposable sets \cite{Fr, RS},
and many others. 
The following notion was introduced by van de Vel \cite{Vel}.

\begin{definition} \label{ConvVel}
A family $\cal{C}$ of subsets of a set $Y$ is called a
{\it convexity} on $Y$\index{convexity on a space}
if it contains $\emptyset$ and $Y$, is closed with respect to intersections of arbitrary subfamilies and is closed with respect to unions of an arbitrary updirected subfamilies.
\end{definition}

Van de Vel \cite{Vel} proved  a Michael's selection type theorem for LSC mappings into completely metrizable spaces $Y$ endowed with a convexity $\cal{C}$ which satisfies a set of assumptions such as compatibility with metric uniformity, compactness and connectedness of 
{\it polytopes}\index{polytope}
(i.e. convex hulls of finite sets), etc.
One of the crucial restriction was the so-called 
{\it Kakutani} $S_4$-property \index{Kakutani $S_4$-property}
which means that every pair of disjoint convex sets
admits extensions up to two complementary convex sets (i.e. half-spaces). 
In the special issue of "Topology and Applications" entirely dedicated
to 50-th anniversary of selection theory and to the 80-th anniversary of Ernest Michael, 
Horvath \cite{HorM} proposed an approach which gives a selection theorem for convexities
with the relative $S_4$-property.

\begin{theorem}[\cite{HorM}] \label{Horth_1}
Let $(Y;\cal{C})$ be a completely metrizable space with convexity, let all polytopes be compact and connected, and let Kakutani $S_4$-property hold on polytopes with respect to the induced convexities. Then every LSC mapping $F:X \to Y$ from a paracompact domain and with closed convex values admits a continuous singlevalued selection
\end{theorem}

He also added facts on selections with results on extensions, approximations and fixed points.
In fact, compactness and connectedness of polytopes together with the Kakutani $S_4$-property imply homotopical triviality of polytopes and moreover, of all completely metrizable convex sets. Therefore  the following Horvath's theorem generalizes the previous one.

\begin{theorem} [\cite{HorM}] \label{Horth_2}
Let $(Y;\cal{C})$ be a completely metrizable space with convexity for which all polytopes are homotopically trivial. Then every LSC mapping $F:X \to Y$ with paracompact domain and with closed convex values admits a continuous singlevalued selection
\end{theorem}

The key technical ingredient proposed in \cite{HorM} was the
{\it van de Vel} property\index{van de Vel property}, 
which roughly speaking, fixes the existence of enough reflexive relations (entourages) 
$R \subset Y \times Y$ such that for every subset $Z \subset Y$ all simplicial complexes
$$
S_R(Z) = \{A \,\, {\rm is\,\,a\,\,finite\,\, subset\,\, of}\, Z:\,\,Z \cap (\cap_{a \in A} R(a)) \not= \emptyset \}
$$
are homotopically trivial.

In the same issue of {\it Topology and its Applications},
Gutev \cite{GuConv} presented results on a somewhat similar matter. Briefly, he proposed another approach to proving the van de Vel selection theorem. He incorporated the proof into the technique of the so-called {\it $c$-structures} which was suggested around
1990 by Horvath \cite{Hor}. A 
$c$-structure\index{$c$-structure}
$\chi$ on a space $Y$ associates to every finite subset $A \subset Y$ some contractible subset $\chi(A) \subset Y$ such that $A \subset B$ implies $\chi(A) \subset \chi(B)$. In the case of the finite subsets $A$ of some prescribed $S \subset Y$ Gutev use the term
{\it $c$-system on $S$}\index{$c$-system}.
Of course, as usual $\chi(\{y\})=\{y\})$. Perhaps the typical statements in \cite{GuConv} are, for example:

\begin{theorem} [\cite{GuConv}]\label{GuConv-1}
Let $X$ be a paracompact space, $Y$ a space, $\chi: Fin(S) \to Y$ a $c$-system on $S \subset Y$, and $G:X \to S$ a Browder mapping\index{Browder mapping} (i.e. a multivalued mapping with all point-preimages open).
Then $conv_{\chi}(G)$ has a continuous singlevalued selection.
\end{theorem}

\begin{theorem} [\cite{GuConv}] \label{GuConv-2}
Let $X$ be a paracompact space, $(Y,\mu, \cal{C})$ a uniform space endowed with a $S_4$ convexity for which all polytopes are compact and all convex sets are connected. Let $\cal{V} \in \mu$ be an open convex cover of $X$ and $F:X \to Y$ a convexvalued LSC mapping. Then $F$ admits a LSC convexvalued selection $\Phi: X \to Y$ such that the family $\{\Phi(x)\}_{x \in X}$ of its values refines $\cal{V}$.
\end{theorem}

By using the latter fact, under the assumptions of the van de Vel theorem, one can construct a decreasing sequence $\Phi_n: X \to Y$ of convexvalued LSC mappings with $\sup_{x \in X}  diam(\Phi_{n}(x)) \to 0,\,\,n \to \infty.$ Therefore  the pointwise passing to $\cap_{n} \{Clos(\Phi_{n}(x))\}$ gives the desired selection of $F=\Phi_0$.

\medskip
{\bf 5.2.} Some results on selections appeared for {\it hyperconvex} range spaces. Recall, that a metric space is 
{\it hyperconvex}\index{hyperconvex space}
if and only if it is injective with respect to extensions which preserves
the modulus of continuity. In more direct terms,
a metric space $(Y,d)$ is hyperconvex if and only if for every family $\{(y_\alpha; r_\alpha) \in Y \times (0;+\infty)\}_{\alpha \in I}$ the inequalities $d(y_\alpha, y_\beta) \leq r_\alpha + r_\beta$ for all $\alpha, \beta \in I$ imply the nonemptiness of the intersection $\cap_{\alpha} D[y_\alpha; r_\alpha]$ of closed balls $D[y_\alpha; r_\alpha]=Clos(D(y_\alpha; r_\alpha))$.

Each hyperconvex space admits the natural 
{\it ball convexity}\index{ball convexity}
in which polytopes are exactly the sets $\cap \{D[y,r]:  A \subset D[y,r]\}$ with finite $A \subset Y$. Some authors \cite{Mar1, Wu} also use the term {\it sub-admissible} for sets which are convex with respect to the ball convexity. For example, Wu \cite[Theorem 2.4]{Wu}
proved a selection theorem for the so-called {\it locally uniform weak LSC} mappings into hyperconvex spaces.  Note that hyperconvex space equipped with the ball convexity is a uniform convex space  with homotopically trivial polytopes, \cite{HorM}. This is why a selection theorem for LSC mappings \cite[Theorem 2.3]{Wu}  is a special case of 
selection theorems for generalized convexities. Markin \cite{Mar1} generalized Wu's result to a 
wider class of multivalued mappings which he named {\it quasi LSC} although this is exactly the class of {\it almost LSC} mappings
introduced by Deutsch and Kenderov \cite{DK}.

Another type of selection theorems for hyperconvex range spaces deals with various Lipschitz-type restrictions on a mappings and selections. A subset $Z \subset (Y;d)$ is said to be
{\it externally hyperconvex}\index{externally hyperconvex set} 
if for any points $x_\alpha \in Z$ and any reals $r_\alpha$ with
$d(x_\alpha, Z) \leq r_\alpha \, \, \hbox{\rm{and}} \, \,  d(x\alpha, x_\beta)\leq r_\alpha + r_\beta,$ the intersection
$\cap_\alpha\,D[x_\alpha, r_\alpha] \bigcap Z$ is nonempty. Khamsi, Kirk and Yanez \cite{KKY} proved the following:

\begin{theorem} [\cite{KKY}] \label{KKYan}
Let $(Y;d)$ be any hyperconvex space, $S$ any set, and $F:S \to Y$ any mapping with externally hyperconvex values. Then there exists a singlevalued selection $f$ of $F$ such that
$$
d(f(x),f(y)) \leq Hausd_{d}(F(x),F(y))\,\,\,x,y \in S.
$$
\end{theorem}

In particular, for a metric domain $S=(M;\rho)$ and for a nonexpansive $F:S \to Y,$ one can assume a selection $f$ of $F$ to be also nonexpansive,
provided that all values $F(x)$ are bounded  externally hyperconvex sets.

These results were applied by Askoy and Khamsi \cite{AsKh} for range spaces which are {\it metric trees}. Briefly, a 
metric tree\index{metric tree}
is a metric space $(Y,d)$ such that for any $x,y \in Y$ there exists a unique arc joining $x$ and $y$ and such that the arc is isometric to a segment on the real line.

\begin{theorem} [\cite{AsKh}]\label{KhAsk}
Let $(Y;d)$ be a metric tree and $F:Y \to Y$ a mapping all of whose values all are bounded closed convex sets. Then there exists a singlevalued selection $f$ of $F$ such that
$$
d(f(x),f(y)) \leq Hausd_{d}(F(x),F(y))\,\,\,x,y \in Y.
$$
\end{theorem}

A somewhat similar result was proved by Markin \cite[Theorem 4.3]{Mar2}:
\begin{theorem} [\cite{Mar2}]\label{KhAsk_1}
Let $X$ be a paracompact space, $(Y;d)$ a complete metric tree and $F:X \to Y$ an almost LSC mapping all
of whose values  are bounded closed convex sets. Then there exists a singlevalued continuous selection $f$ of $F$.
\end{theorem}

Selection theorems with respect to various types of convexities, $L$-structu\-res, and
$G$-structures, were obtained in \cite{Di,Li,YD}, etc. As a rule, all results here are special cases or versions of van de Vel's convexities, or Horvath's convexities.

\medskip
{\bf 5.3.} As for some other ``inner convexities'', in a series of papers
de Blasi and Pianigiani  studied multivalued mappings into so-called 
{\it $\alpha$-convex} metric spaces $(Y;d)$\index{$\alpha$-convex metric space}. 
This means the existence a continuous mapping $\alpha: Y \times Y \times [0;1] \to Y$ with natural restrictions
$$\alpha(y,y,t)=y,\, \,
\alpha(y,z,0)=y,\,\, \, \alpha(y,z,1)=z$$ and with assumption that for some suitable $r=r(\alpha)>0$ and for every $\varepsilon < r$ there exists $0<\delta \leq \varepsilon$ such that the inequality
$$
Hausd(\{\alpha(y,z,t): t \in [0;1]\}, \{\alpha(y',z',t): t \in [0;1]\}) < \varepsilon
$$
for Hausdorff distance between curvilinear segments holds for each $(y,y') \in Y^2,\, (z,z') \in Y^2$ with $d(y,y')<\varepsilon$ and $d( z,z') < \delta$.
Clearly, the last assumption reminds one of
the estimate for
$d(\alpha(y,z,t), \alpha(y',z',t))$ from Michael's geodesic structure \cite{MConv}.
As usual, $C \subset Y$ is convex (with respect to $\alpha: Y \times Y \times [0;1] \to Y$) if
$\{\alpha(y,z,t): t \in [0;1]\} \subset C$ provided that $y \in C$ and $z \in C$.

\begin{theorem} [\cite{BlP-5}]\label{Blasi_1}
Let $X$ be a paracompact space and $Y$ an $\alpha$-convex complete metric space. Then every LSC mapping $F:X \to Y$ with closed convex values admits a continuous singlevalued selection
\end{theorem}

As it was shown earlier in \cite{BlP-4} for the case of compact $X$ and $Y$, and $dim X < \infty,$ Theorem 43 is true for every  $\alpha: Y \times Y \times [0;1] \to Y$ with
$\alpha(y,y,t)=y,\, \alpha(y,z,0)=y,\,\,\alpha(y,z,1)=z$. A set of applications in fixed point theory and in degree theory are also presented in \cite{BlP-9, BlP-5}. Kowalska \cite{Kow} proved a theorem which unifies selection Theorem 42 with a graph-approximation theorems in the spirit of Ben-El-Mechaiekh and
Kryszewski \cite{BMKr} who considered the case of classical convexity of a Banach space.

Among others results let us mention the paper of Kisielewicz \cite{Kis} in which he used a convexity structures in functional Banach spaces $C(S,R^n)$  and $L^{\infty}(T,R^n)$ of all continuous mappings over a compact Hausdorff domain $S$ and all equivalence classes of almost everywhere bounded mappings over a measure space $(T, \mu)$.
Both of these convexities remind of the notion of 
{\it decomposable} set\index{decomposable set}
of functions \cite{Fr,RS}.

\begin{definition}
\begin{description}
\item{(1)} A subset $E \subset L^{\infty}(T,R^n)$ is said to be {\it decomposable}
if $\chi_A f+\chi_{T \setminus}g$ belongs to $E$ provided that $f \in E, g \in E$ and $A$ is a measurable subset of $T$.
\item{(2)} A subset $E \subset L^{\infty}(T,R^n)$ is said to be $L$-{\it convex}\index{$L$-convex set} if $pf+(1-p)g$ belongs to $E$ provided that $f \in E, g \in E$ and $p:T \to [0;1]$ is a measurable function.
\item{(3)} A subset $E \subset C(S,R^n)$ is said to be $C$-{\it convex}\index{$C$-convex set} if $hf+(1-h)g$ belongs to $E$ provided that $f \in E, g \in E$ and $h:T \to [0;1]$ is a continuous function.
\end{description}
\end{definition}

\begin{theorem}[\cite{Kis}] \label{Kis_1}
\item{(1)} Let $X$ be a paracompact space and $F:X \to C(S,R^n)$ a LSC mapping with closed $C$-convex values. Then $F$ admits a continuous singlevalued selection if and only if its $n$-th derived mapping $F^{(n)}$ has nonempty values.
\item{(2)} Same as $(1)$ but for $F:X \to L^{\infty}(T,R^n)$ with closed $L$-convex values.
\end{theorem}

Note that for closed subsets of $L^{\infty}(T,R^n)$ their $L$-convexity coincides with decomposability plus usual convexity
\cite[Proposition 4]{Kis}. Recall that
the
{\it derived mapping}\index{derived mapping}
$F^{(1)}(x)$ of a multivalued mapping $F:X \to Y$ is defined by setting
$$
F^{(1)}(x) = \{y \in F(x):  (x' \to x) \Rightarrow dist(y,F(x')) \to 0 \} \subset F(x),\,\,x \in X
$$
and $F^{(k+1)}(x)=(F^{(k)})^{(1)}(x)$. Also, a
well-known result of Brown \cite{Br} states that for a convexvalued map
$F:X \to R^n,$ the nonemptiness of all $F^{(n)}(x), x \in X,$ implies that $F^{(n)}: X \to R^n$ is a
LSC selection of $F$. Hence the standard selection techniques can be applied to $F^{(n)}$. Rather simple examples show the essentiality of the
finite-dimensionality of the range space.

\medskip
{\bf 5.4.} Based on the ingenious idea of Michael who proposed in
\cite{MPar} the notion of a
{\it paraconvex set}\index{paraconvex set},
the authors in \cite{
RSPol,
RSUn, 
RSNei,
RSMin,
S1, 
S2} 
systematically studied  another approach to weakening (or, controlled omission) of convexity.
Roughly speaking, to each closed subset $P \subset B$ of a Banach space one associates
a
numerical function, say $\alpha_{P}:(0,+\i) \to [0,2)$, the so-called
function of nonconvexity of  $P$. The identity $\alpha_{P}\equiv 0$ is
equivalent to the convexity of $P$ and the more $\alpha_{P}$ differs
from zero the "less convex"\,is the set $P$ .

\begin{definition} The 
{\it function of nonconvexity}\index{function of nonconvexity}
$\alpha_{P}(\cdot)$ of the set $P$ associates to each number $r > 0$
the supremum of the set
$$\lbrace \sup \lbrace dist(q,P)/r \, : \, q \in conv(P \cap D_r) \rbrace \rbrace$$
over all open balls $D_r$ of the radius $r$.
\end{definition}

For a function $\alpha: (0;+\infty) \to (0;+\infty)$ a nonempty closed subset $P$ of a Banach
space is said to be 
$\alpha$-paraconvex\index{$\alpha$-paraconvex set} 
provided that function
$\alpha(\cdot)$ pointwisely majorates the function of nonconvexity
$\alpha_{P}(\cdot)$. Then $P$  is said to be paraconvex provided that $\sup \alpha_{P}(\cdot) < 1$.

In \cite{RSMin} we proved a paraconvex version of the Ky Fan-Sion minimax theorem.

\begin{theorem} [\cite{RSMin}] \label{MinMax}
Let $\alpha :(0,\infty) \to (0,1)$
be a function with the right upper limits less than 1 over the
closed ray $[0,\infty)$. Let $f: X \times Y \to {\mathbb R}$ be a real
valued function on Cartesian product of two $AR$-subcompacta $X$ and
$Y$ of a Banach spaces and suppose that:
\begin{description}
\item{(1)} For each $c \in {\mathbb R}$ and each $x_{0} \in X$ the
set $\lbrace y \in Y : \enspace f(x_0,y) \leq c \rbrace$ is
$\alpha$-paraconvex compact; and
\item{(2)} For each $d \in R$ and each $y_{0} \in Y$ the
set $\lbrace x \in X : \enspace f(x,y_{0}) \geq d \rbrace$ is
$\alpha$-paraconvex compact for a fixed $\alpha :(0,\infty) \to [0,1)$.
Then $$\max_{X} (\min_{Y}f(x,y)) = \min_{Y}( \max_{X}f(x,y)).$$
\end{description}
\end{theorem}

It is interesting to note that our minimax theorem includes cases when
finite intersections
$$\bigcap_{i=1}^{n}
\lbrace \lbrace x \in X :  f(x,y_i) \geq c \rbrace : y_i \in Y
\rbrace,\quad  \bigcap_{j=1}^{k} \lbrace \lbrace y \in Y : f(x_j,y)
\leq d \rbrace : x_j \in X \rbrace
$$
of sublevel and uplevel sets are possibly nonconnected: intersection of two paraconvex sets can be nonconnected.

Usually a proof of a minimax theorem for generalized convexities exploits the
{\it intersection} property of convex sets and reduces minimax theorem to some kind of  KKM-principle. In our case we used the fact that the closure of {\it unions} of
directly ordered family of arbitrary paraconvex sets are also paraconvex.
Therefore  as a base for obtaining minimax theorem we have used the
selection theory of multivalued mappings instead of versions of the
KKM-principle.

In \cite{RSRet} we examined the following natural question:
{\it Does paraconvexity of a set with respect to the classical
convexity structure coincide with convexity under some generalized
convexity structure?} In other words,
is paraconvexity a real nonconvexity, or is it maybe a kind of some generalized convexity?
It turns out that sometimes the answer is affirmative.

\begin{theorem} [\cite{RSRet}] \label{ContRetr}
Let $0\leq \alpha < 0,5$ and $F: X \to
B$ be a continuous multivalued mapping of a paracompact
space $X$ into a Banach space $B$ such that all values $F(x)$
are bounded $\alpha$-paraconvex sets. Then there
exists a continuous singlevalued mapping $\mathfrak{F}:X \to
C_{b}(B,B)$ such that for every $x \in X$ the mapping
$\mathfrak{F}_{x}:B \to B$ is a continuous retraction of $B$ onto the
value $F(x)$ of $F$.
\end{theorem}

Here $C_{b}(B,B)$ denotes the Banach space of all continuous bounded mappings of a Banach space $B$ into itself. The key point of the proof is that the set $URetr_P \subset C_{b}(B,B)$ of all uniform retractions onto an $\alpha$-paraconvex set $P \subset B$ is a $\frac{\alpha}{1-\alpha}$-paraconvex subset of $C_{b}(B,B)$.

As a corollary, by continuously choosing a retraction onto a paraconvex sets, we showed that if in addition all values $F(x),\,\,x \in X$, are pairwise disjoint
then the metric subspace $Y= \bigcup_{x \in X}F(x) \subset B$ admits a
convex metric structure $\sigma$ such that each value $F(x)$ is convex with respect to
$\sigma$.

Finally, let us mentioned the result of Makala \cite[Theorem 3.1]{Maka} who obtained the selection theorem for LSC mappings $F:X \to Y$ from a collectionwise normal domains $X$ such that each value $F(x)$ equals to $Y$ or, is a compact paraconvex subsets of $Y$. The main difficulty here was that the class of such values in general,
is not  closed with respect to intersections with balls.

\medskip
{\bf 5.5.}    Yet another type of a controlled "nonconvexity" which is in some sense intermediate
between paraconvexity and Menger's metric convexity goes back to Vial \cite{Vi} and during the last decade was intensively studied in \cite{BalIv,Iv-1,Iv-2}.

For every two points $x$ and $y$ of a normed space $(Y;\|\cdot \|)$ and for every $R\geq 0,5\|x-y \|$\,
denote by $D_{R}[x;y]$ the intersection of all closed $R$-balls containing $x$ and $y$. Clearly,
when $R \to +\infty,$ such set $D_{R}[x;y]$ tends (with respect to the Hausdorff distance)
to the usual segment $[x;y]$.

\begin{definition}[\cite{Vi}]
A subset $A$ of a  normed space $( Y ; \| \cdot \| )$ is said to be
{\it weakly convex w.r.t $R>0$}\index{weakly convex set}
if for every $x, y \in A,$ with $0<\|x-y\|<2R,$ there exists a point $z \in A \cap D_{R}[x;y]$ that differs from $x$ and $y$.
\end{definition}

In a Hilbert space $H$ the metric projection $P_A$ of an $R$-neighborhood of a weakly convex w.r.t. $R$ set $A$ is singlevalued. The set $\{(x;y) \in \mathbb{R}^2:  x\geq 0\,\, {\rm or}\,\, y\geq 0\}$ is $(\sqrt{2}/2-)$paraconvex but is not weakly convex with respect to arbitrary $R>0$. The set
$\{x \in \mathbb{R}^n: \|x\|\geq R\}$  is not paraconvex and is weakly convex w.r.t $R$. However, sometimes weak convexity implies paraconvexity.  For example, 
in a Hilbert space $H$, if
$z \in H$ and $0<r<R$, then every weakly convex (w.r.t. $R$) subset $A \subset D(z;r)$
is $(r/R)$-paraconvex \cite{Iv-1}.

\begin{theorem}[\cite{Iv-2}]
Let $X$ be a paracompact space and $0< \varepsilon <R$. Then for every continuous singlevalued
$\varepsilon$-selection $f_{\varepsilon}:X \to H$ of a LSC mapping $F:X \to H$ with closed
and weakly convex (w.r.t. $R$) values
there exists a continuous singlevalued selection.
\end{theorem}

\begin{theorem}[\cite{Iv-2}]
Let $X$ be a paracompact subset of a topological vector space $Z$ and a uniformly functional contractible subset of $Z$. Then for every $R>0$,
each Hausdorff uniformly continuous
mapping  $F:X \to H$
with closed and weakly convex (w.r.t. $R$) values admits a continuous singlevalued selection.
\end{theorem}

Both contractibility and uniform continuity restrictions are essential as Examples 1-3 from \cite{Iv-2} show.

\medskip
{\bf 5.6.} To complete this
section we mention one more nontrivial convexity structure. Namely, the so called {\it tropical (or, max-plus)} geometry in $(\R \cup \{-\infty\}^n$. It is very intensively studied area with many various applications in abstract convex analysis, algebraic geometry, combinatorics, phylogenetic analysis, etc. For a survey and details cf.  \cite{Wiki}.

\begin{definition}
For an ordered $N$-tuple $t=(t_j)$ of "numbers" $t_j \in [-\infty; 0]$ with $\max\{t_j\} = 0$ and
for a points $A_1(x_1^1, x_2^1,....,x_n^1),..., A_N(x_1^N, x_2^N,....,x_n^N)$ from $(\R \cup \{-\infty\})^n$ their max-plus $t$-combination is defined as the point
$\left(\max_j \{x_1^j + t_j\}, \max_j \{x_2^j + t_j\} ..., \max_j \{x_n^j + t_j\} \right)$.
A subset $C \subset (\R \cup \{-\infty\})^n$ is said to be 
{\it max-plus convex}\index{max-plus convex set}
if it contains all max-plus $t$-combinations of all of
its points.
\end{definition}

Zarichnyi proved a selection theorem for max-plus convexvalued mappings.

\begin{theorem}[\cite{Zar}]
Let $X$ and $Y$ be compact metrizable spaces and $Y \subset \R^n$. Then every LSC mapping
$F:X \to Y$ with max-plus convex values admits a continuous singlevalued selection.
\end{theorem}

It is interesting to observe
that the proof never uses any sequential procedure of approximation.
Instead, Zarichnyi constructs a version of Milyutin surjection $M:Z \to X$ of a zero-dimensional compact space $Z$ onto $X$ and associating map $m:X \to I(Z)$ with values in {\it idempotent} probability
measures. Next, exactly as in \cite{RSS}, the desired selection $f$ of $F$ is defined as the idempotent barycenter $f(x)= \int_{M^{-1}(x)} s(t)dm(x)$ for a suitable selection
$s$ of the composition $F \circ M$ (such a selection exists due to Theorem 2).

\section{Multivalued selections}
\label{sec:6}

The foundation
for results of this section is the compactvalued selection Theorem 3.
Historically there were various ways to prove this result or its variants: 
the original Michael's
approach \cite{MultSel} via pointwise closures of limit point sets of certain "tree" of 
$2^{-n}$-singlevalued selections,
Choban's method of coverings \cite{ChobGen} which axiomatized and transformed Michael's construction into a maximally possible general form, approach based on the notion of (complete)\,
{\it sieves} \cite{ChKenR}, and a proof via the $0$-dimensional selection theorem \cite{RS}.

In a series of papers Gutev recently proposed a more advanced point of view for sieves on a set $X$. Recall that a 
{\it tree}\index{tree}
is a partially ordered set $(T;\preceq)$ with all well-ordered sublevel
sets $\{a: a \preceq b, a\neq b\}_{b \in T}$. Roughly speaking, in \cite{GuSie} a sieve\index{sieve} on $X$ is defined as  some kind of multivalued mapping $(T; \preceq) \to X$ which is order-preserving with respect to inverse inclusion. 
In particular, as a corollary of his techniques,
Gutev \cite[Corollary 7.3]{GuSie}
obtained the following generalization 
of compactvalued selection Theorem 3,
which was   proposed earlier
in \cite{AllCal}.

\begin{theorem} [\cite{GuSie}]\label{Seive_1}
A multivalued mapping $F:X \to Y$
admits a compactvalued USC selection $H:X \to Y$, which, in turn, admits a compactvalued LSC selection $G:X \to Y$, provided
that the following conditions are satisfied:
\begin{description}
\item{(1)} $X$ is a paracompact space;
\item{(2)} $Y$ is a monotonically developable and sieve-complete space;
\item{(3)} $F$ is a LSC mapping; and
\item{(4)} For every $x \in X$, $F(x)$ is a closed subset of $Y$.
\end{description}
\end{theorem}

Here, in comparison with compactvalued selection Theorem 3, only the restriction $(2)$ is changed. Monotonically developable spaces are a natural generalization of Moore spaces. Note that $Y$ is monotonically developable and sieve-complete space if and only if $Y$ is the image of a completely metrizable non-Archimedean space under some open surjection \cite{Wicke}. 
If  one omits in $(2)$
the assumption that $Y$ is a monotonically developable then by
\cite[Corollary 7.2]{GuSie} it is possible to guarantee only the existence an USC compactvalued selection $H:X \to Y$. 

%  Similarly, for collectionwise normal domains Theorem 49 is true, provided that for every $x \in X$ either the  value $F(x)$ is compact or $F(x)=Y$. For weakly paracompact domains Theorem 49 holds if one restricts only to LSC compactvalued selection $G:X \to Y$. 
If  we equip a paracompact domain $X$ in Theorem 49
by a sequence $\{X_n\}$ of its finite-dimensional subspaces
$dimX_n \leq n$, then by
 \cite[Corollary 7.7]{GuFin}, 
we always obtain an USC compactvalued selection
$H:X \to Y$ with $|H(x)| \leq n+1,\,\,x \in X_n$.
If we eqip a paracompact domain $X$ by a sequence $\{X_n\}$ of its finite-dimensional subspaces, $dimX_n \leq n$, then we always obtain an USC compactvalued selection $H:X \to Y$ with $|H(x)| \leq n+1,\,\,x \in X_n$.

As a rule, all selections in \cite{GuSie} are constructed as a composition of two suitable multivalued mappings. The first one is related to completeness and the other one arises from a system (tree) of various coverings of the domain and their refinements.

Applying the same "trees-sieves"\,technique in \cite{GuClos} upper semicontinuity of a selection was replaced by closedness of its graph. Below are two typical examples.

\begin{theorem} [\cite{GuClos}] \label{ClosedGr}
For a $T_1$-space $X$ the following statements  are equivalent:
\begin{description}
\item{(1)} $X$ is normal;
\item{(2)} If $Y$ is a metrizable space and $F:X \to Y$ is a compactvalued LSC mapping then there are compactvalued mappings $G:X \to Y$ and $H:X \to Y$ such that $G(x) \subset H(x) \subset F(x), \,x \in X$,\, $G$ is LSC and the graph of $H$ is a closed subset of $X \times Y$;
\item{(3)} If $Y$ is a metrizable space and $F:X \to Y$ is a compactvalued LSC mapping then there exists a compactvalued selection of $F$ with a closed graph.
\end{description}
\end{theorem}

\begin{theorem} [\cite{GuClos}] \label{ClosedGr_2}
For a $T_1$-space $X$ the following statements  are equivalent:
\begin{description}
\item{(1)} $X$ is countably paracompact and normal;
\item{(2)} Same as $(2)$ in theorem above, but for closedvalued mappings into a separable range space.
\end{description}
\end{theorem}

Choban, Mihaylova and Nedev \cite{CMN} collected various types of selection characterizations of classes of topological spaces formulated in terms of multivalued selections. Recall that the $n$-th image of a set is defined inductively by setting $F^1(A)=F(A)$, $F^{n+1}(A)=F(F^{-1}(F^{n}(A)))$,\, and that the {\it largest image} is defined as the union of all $n$-th images, $n \in \N$.

\begin{theorem} [\cite{CMN}] \label{ChMihNed}
For a $T_1$-space $X$ the following statements  are equivalent:
\begin{description}
\item{(1)} $X$ is strongly paracompact (i.e Hausdorff and each open cover admits a star-finite refinement);
\item{(2)} For every LSC mapping $F:X \to Y$ into a discrete space $Y$ there exists a discrete space $Z$, a singlevalued map $g:Z \to Y$, a LSC mapping $G:X \to Z$, and an USC finitevalued mapping $H:X \to Z$ such that $g(G(x)) \subset g(H(x)) \subset F(x), \,\,x \in X$,\, and all sets $H^{\infty}(x)$ are countable;
\item{(3)} Same as $(2)$ but without $LSC$ mapping $G$ and without finiteness of the values $H(x)$;
\item{(4)} Same as $(2)$ but with a regular $X$, without $USC$ mapping $H$, and with a countable $G^{\infty}(x)$.
\end{description}
\end{theorem}

For a space $X$, let $c\omega(X)$ denote the 
{\it cozero dimensional kernel}\index{cozero dimensional kernel} 
of $X$, i.e. the complement of the union of all open zero-dimensional subsets of $X$.

\begin{theorem} [\cite{CMN}] \label{ChMihNed_2}
For a $T_1$-space $X$ the following statements  are equivalent:
\begin{description}
\item{(1)} $X$ is strongly paracompact and $c\omega(X)$ is Lindel\"{o}f ;
\item{(2)} See $(2)$ in previous theorem with $Y=Z$ and $g=id|_Z$;
\item{(3)} Same as $(2)$ but without $LSC$ mapping $G$ and without finiteness of values $H(x)$;
\item{(4)} Same as $(2)$ but with a regular $X$, without $USC$ mapping $H$ and with a countable $G^{\infty}(x)$.
\end{description}
\end{theorem}

In the next theorem $l(X)$ denotes the Lindel\"{o}f number of the space $X$ and singlevalued selections are not assumed to be continuous.

\begin{theorem} [\cite{CMN}] \label{ChMihNed_3}
For any regular space $X$ and any cardinal number $\tau$ the following statements  are equivalent:
\begin{description}
\item{(1)} $l(X) \leq \tau$;
\item{(2)} For every LSC closedvalued mapping $F:X \to Y$ into a complete metrizable space $Y$ there exists a LSC closedvalued selection $G$ of $F$ such that $l(G(X))\leq \tau$;
\item{(3)} For every LSC mapping $F:X \to Y$ into a complete metrizable space $Y$ there exists a singlevalued selection $g$ of \,\,$Clos(F)$ such that $l(g(X))\leq \tau$;
\item{(4)} For every LSC mapping $F:X \to Y$ into a discrete space $Y$ there exists a singlevalued selection $g$ of \,\,$Clos(F)$ such that $|g(X)|\leq \tau$;
\item{(5)} Every open cover of $X$ admits a refinement of cardinality $\leq \tau$.
\end{description}
\end{theorem}

Similar characterization  was obtained in \cite{CMN} for the degree of compactness of a space. Gutev and Yamauchi in \cite{GuY}, using once again the "trees-sieves"\,technique, presented a generalizations of results \cite{CMN} to arbitrary complete metric range spaces. For example, in comparison with Theorem 52 they proved the following:

\begin{theorem} [\cite{GuY}] \label{GutYam}
For a $T_1$-space $X$ the following statements  are equivalent:
\begin{description}
\item{(1)} $X$ is strongly paracompact;
\item{(2)} For every LSC closedvalued mapping $F:X \to Y$ into a complete metric space $(Y;\rho)$ there exist a complete ultrametric space $(Z;d)$, a uniformly continuous map $g:Z \to Y$, and a USC compactvalued mapping $H:X \to Z$ such that $g(H(x)) \subset F(x),\,\, x \in X$, and the set $H(H^{-1}(S))$ is totally $\varepsilon$-bounded whenever $\varepsilon>0$ and $S \subset Z$ is totally $\varepsilon$-bounded;
\item{(3)} For every LSC closedvalued mapping $F:X \to Y$ into a discrete space $Y$ there exist a discrete space $Z$, a singlevalued map $g:Z \to Y$, and a USC compactvalued mapping $H:X \to Z$ such that $g(H(x)) \subset F(x), \,\,x \in X$ and the set $H(H^{-1}(S))$ is finite whenever $S \subset Z$ is finite.
\end{description}
\end{theorem}

Similarly, \cite[Corollaries 6.2 and 6.3]{GuY} a space $X$ is  strongly paracompact and $c\omega(X)$ is Lindel\"{o}f (resp.,  compact) if and only if for every LSC closedvalued mapping $F:X \to Y$ into a completely metrizable space $Y$ there exists a USC compactvalued selection $H$ of $F$ such that set $H(H^{-1}(S))$ is separable (resp., compact), whenever $S \subset Z$ is separable (resp., compact).

Yamauchi \cite{YH} gave a selection characterization of the class which unifies compact spaces and finite-dimensional paracompact spaces. A topological space is said to be 
{\it finitistic}\index{finitistic space}
(another term is {\it boundedly metacompact})\, if  any of its 
open covers admits an open refinement of finite order, or equivalently for paracompact spaces, if and only if it contains a compact subset $K$ such that each closed subset of the complement of $K$ is finite-dimensional. Below is a typical statement.

\begin{theorem} [\cite{YH}] \label{Yam}
For a $T_1$-space $X$ the following statements  are equivalent:
\begin{description}
\item{(1)} $X$ is paracompact and finitistic;
\item{(2)} Each LSC closedvalued mapping $F:X \to Y$ into a completely metrizable space $Y$ admits a USC compactvalued selection $H:X \to Y$ of $F$ with the property that for every open cover $\nu$ of $Y$ the exists a natural number $N$ such that every value $H(x), \,x \in X$, can be covered by some $\nu_0 \subset \nu$ with $Card(\nu_0)\leq N$.
\end{description}
\end{theorem}

Finally, let us mentioned the survey paper  by Choban \cite{Chob}
on reduction principles in selection theory. Briefly, he discussed  questions concerning 
extensions of LSC mappings $F$ with nonparacompact domains onto paracompact ones and a notion of (complete) metrizability of family $\{F(x)\}$
of values rather than (complete) metrizability of a range space.

\section{Miscellaneous results}
\label{sec:7}

{\bf 7.1.} Tymchatyn and Zarichnyi \cite{TZ} applied the selection theorem of Fryszkow\-ski \cite{Fr} to
decomposable-valued mappings $F: X \to L_{1}([0;1], B)$ in order to construct a continuous linear regular operator which extends partially defined pseudometrics to pseudometrics defined on the whole domain. Denote by $\cal{P}\cal{M}$$(X)$ the set of all continuous pseudometrics over metrizable compact $X$ and $\cal{P}\cal{M}$ the subset of $\cal{P}\cal{M}$$(X)$  of all continuous pseudometrics $\rho$ with compact domains $dom(\rho) \subset X$. Identifying a pseudometric with its graph we can consider both of these sets endowed with the topology induced from the compact exponent $exp(X \times X \times [0;\infty) )$.

\begin{theorem} [\cite{TZ}] \label{TymZar_1}
There exists a continuous linear regular extension operator $u: \cal{P}\cal{M} \to \cal{P}\cal{M}$ $(X)$,\quad $u(\rho)|_{dom(\rho) \times dom(\rho)} \equiv \rho$.
\end{theorem}

Here, regularity of an operator means that it preserves the {\it norm} of pseudometrics, i.e. their maximal values. As it typically arises for extensions, the answer is given by some {\it formula}. Namely, under some isometric embedding $X$ into a separable Banach $B$, the desired operator $u$ can be defined as
$$
u(\rho)(x,y) = \int_{0}^{1}\,\,\rho(f(dom(\rho),x)(t),f(dom(\rho),y)(t))\,\,dt,
$$
where $f: expX \times X \to L_1([0;1],B)$ is a continuous singlevalued selection of the decomposable-valued LSC mapping $F: expX \times X \to L_1([0;1],B)$ defined by
$$
F(A,x) = L_1([0;1],\{x\})=\{x\}, x \in A;\qquad F(A,x) = L_1([0;1],A),x \notin A.
$$

Metrizability of a compact space $X$ is a
strongly essential assumption \cite[Theorem 6.1]{TZ}.

\begin{theorem} [\cite{TZ}] \label{TymZar_2}
For a compact Hausdorff space $X$ the following statements  are equivalent:
\begin{description}
\item{(1)} $X$ is metrizable;
\item{(2)} There exists a continuous extension operator $u: \cal{P}\cal{M} \to \cal{P}\cal{M}$$(X)$.
\end{description}
\end{theorem}

\medskip
{\bf 7.2.} Gutev and Valov\cite{GVPr} applied selection theory to obtain a new proof of Prokhorov's theorem and its generalization outside the class of Polish spaces. Recall that a probability measure $\mu$ on a $T_{3,5}$-space $X$ is a countably additive mapping $\mu : \cal{B}$ $(X) \to [0,1]$ with $\mu(X)=1$ and with regularity (or, the Radon) property that
$$
\mu(B)=\sup \{\mu(K):\,\,K \subset B,\,\, K\, {\rm is\,\,compact}\}
$$
for every Borel set $B \in \cal{B}$ $(X)$. Roughly speaking, values of measure are
realized over subcompacta with any precision.

The set $P(X)$ of all probability measures can naturally  be considered as the subset of the conjugate space $C^{*}(X)$ of the Banach space $C(X)$ and is endowed with the induced topology. Thus
Prokhorov's theorem states that for a Polish space $X,$ the Radon property holds not for a unique measure but for an arbitrary compact set of measures. Namely, for  a compact $M \subset P(X)$ and for any $\varepsilon > 0,$ there exists a compact $K \subset X$ such that $\mu(X \setminus K) < \varepsilon,$ for all $\mu \in M$. For more general domains  the following result holds \cite{GVPr}:

\begin{theorem} [\cite{GVPr}] \label{Prokh}
For a sieve-complete space $X$, a paracompact space $S \subset P(X)$ and any $\varepsilon > 0,$ there exists a USC compactvalued mapping $H: S \to P(X)$ such that $\mu(X \setminus H(\mu)) < \varepsilon$, for every $\mu \in S$.
\end{theorem}

Note that for a paracompact space $X,$ its sieve-completeness coincides with its Czech-completeness. Returning to the original case of Polish space $X,$ the outline of the proof looks as follows. First, for each $\mu \in P(X)$ the set $G_{\varepsilon}(\mu) = \{K - {\rm compact}: \mu(X \setminus K) < 0,5\varepsilon\} \subset P(X)$ is nonempty simply due to the Radon property. Denote by $exp(X)$ the completely metrizable space of all subcompacta of $X$ endowed with the Vietoris topology, or with Hausdorff distance metric. It turns out that the multivalued mapping $G_{\varepsilon}: P(X) \to exp(X)$ is a LSC mapping. Hence its pointwise closure $F_{\varepsilon}$ is also LSC and moreover,
$\mu(X \setminus K) \leq 0,5\varepsilon,$ for every $\mu \in P(X)$ and every $K \in F_{\varepsilon}(\mu)$. By the compactvalued selection Theorem 3,
the mapping $F_{\varepsilon}$ admits an USC compactvalued selection, say $H: P(X) \to exp(X)$. Finally, the union $\cup \{K' : K' \in H(\mu)\}$ yields
the desired compact subset $K \subset X$.

\medskip
{\bf 7.3.} Zippin \cite{Z} considered the convexvalued selection Theorem 1 as the base for resolving the extension problem for operators from a linear subspaces $E$ of $c_0$ into the spaces $C(K)$, where $K$ is a compact Hausdorff space. For any $\varepsilon > 0,$ he considered the multivalued mapping $F: Ball(E^{*}) \to (1+\varepsilon)Ball(c_{0}^{*})$ by setting $F(e^{*})$ equal to $\{0\}$ if $e^{*}=0$ and
$$
F(e^{*}) = \{  x^{*} \in (1+\varepsilon)Ball(c_{0}^{*}):\,\,x^{*} \,\,{\rm extends}\,\, e^{*}\,\, {\rm and}\,\, \| x^{*} \| < (1+\varepsilon)\| e^{*} \| \}
$$
otherwise.
Under the weak-star topology, all values of $F$ are convex metrizable compacta. After (a nontrivial)
verification
of lower semicontinuity of $F$ and applying Theorem 1, one finds a singlevalued weak-star continuous
mapping $f: Ball(E*) \to (1+\varepsilon)Ball(c_{0}^{*})$ such that $f(e^{*})$ extends $e^{*}$ and
$\| f(e^{*}) \| \leq (1+\varepsilon)\| e^{*} \|$.

Hence for an operator $T: E \to C(K)$ with the norm $\| T \|=1,$ let $f_T: K \to Ball(E^{*})$ be defined by $f_{T}(k)(e)=T(e)(k),$ for $k \in K$. Then $f_T$ is weak-star continuous and hence the composition $f \circ f_T: K \to (1+\varepsilon)Ball(c_{0}^{*})$ with the above selection $f$ is also continuous. Defining $\widehat{T}: c_0 \to C(K)$ by $\widehat{T}(x)(k) = (f \circ f_T)(k)(x),\,\, x\in c_0,$ we see that $\widehat{T}$ extends $T$ because $f(e^{*})$ extends $e^{*}$, $\widehat{T}$ is linear and well-defined, since $(f \circ f_T)(k)$ is a linear functional and
$
\| \widehat{T} \| = \sup\{\widehat{T}(x)(k): \|x\| \leq 1, k \in K\} \leq
\sup\{\|(f \circ f_T)(k)\|\,\,\|x\|: \|x\| \leq 1, k \in K\} \leq 1+\varepsilon.
$

\medskip
{\bf 7.4.} For two multivalued mappings $F_{1}:X \to Y_{1}$,
$F_{2}:X \to Y_{2}$ and a singlevalued mapping $L:Y_{1} \times
Y_{2} \to Y$ denote by $L(F_{1};F_{2})$ the composite mapping,
which associates to each $x \in X$ the set
$$
\{y \in Y:y=L(y_{1};y_{2}),\,\,y_{1} \in F_{1}(x),\,\,y_{2} \in
F_{2}(x)\}.
$$

\begin{definition} Let $f$ be a selection of the composite
mapping $L(F_{1};F_{2})$. Then pair $(f_{1},f_{2})$ is said to be a
{\it splitting}\index{splitting} 
of $f$ if $f_{1}$ is a selection of $F_{1}$,
$f_{2}$ is a selection of $F_{2}$ and $f=L(f_{1};f_{2})$.
\end{definition}

Therefore  the {\it splitting problem} \cite{RSSplit} for the
triple $(F_{1}, F_{2}, L)$ is the problem of finding continuous
selections $f_{1}$ and $f_{2}$ which split a continuous selection
$f$ of the composite mapping $L(F_{1};F_{2})$.
As a special case of a constant mappings $F_{1} \equiv A \subset Y, F_{2} \equiv B \subset Y$ and $L(x_1,x_2)=x_1 + x_2$ we have the following:

\begin{problem} [\cite{RSSplit}]
Let $A$ and $B$ be closed convex subsets of a Banach space $Y$ and $C=A+B$ their Minkowski sum. Is it possible to find continuous singlevalued mappings $a: C \to A$ and $b:C \to B$ such that $c=a(c)+b(c)$ for all $c \in C$?
\end{problem}

The answer is positive for strictly convex and finite-dimensional $A$ and $B$ (cf.  \cite[Corollary 3.6]{RSOpen}), and for finite-dimensional $A$ and $B$ with $C=A+B$ being of a special kind, the so-called $P$-set (cf.  \cite[Theorem 2.6]{BalR}). A collection of various examples and affirmative results on approximative splittings, uniformly continuous (or Lipschitz) splittings can be found in \cite{Bal,BalR,BalRU}. Note that under the replacement of the sum-operator $L(x_1,x_2)=x_1 + x_2$
by an arbitrary linear operator $L,$ the problem has a negative solution even in a rather low-dimensional situation \cite[Example 3.2]{RSOpen}.

\begin{theorem} [\cite{RSOpen}]
For any 2-dimensional cell $D$ there exist:
\begin{description}
\item{(a)} Constant multivalued mappings $F_{1}:D \to R^{3}$
and  $F_{2}:D \to R$ with convex compact values;
\item{(b)} A linear surjection $L: R^{3} \oplus R \to R^{2}$; and
\item{(c)} A continuous selection $f$ of the composite mapping
$F=L(F_{1},F_{2})$, such that $f\neq L(f_{1},f_{2})$ for any
continuous selections $f_{i}$ of $F_{i}$, i=1,2.
\end{description}
\end{theorem}

The construction uses the convex hull $C$ of one full rotation
of the spiral
$
K=\{(\cos t, \sin t,t):0\leq t \leq 2\pi\}
$
and the fact that its projection onto the $xy$-plane admits no continuous singlevalued selections.

\medskip
{\bf 7.5.}  As for the finite-dimensional selection Theorem 4,  during the discussed period three
volumnuous papers were devoted to its versions, generalizations or applications. In \cite{GVDen}
Gutev and Valov proved the following result on the density of selections.

\begin{theorem} [\cite{GVDen}]
Let for a mapping $F:X \to Y$ all assumptions of Theorem 4 be true. Let in addition $\Psi:X \to Y$ be a mapping with an $F_\sigma$-graph such that for each $x \in X$ the intersection
$F(x) \cap \Psi(x)$ is a $\sigma Z_{n+1}$-subset of the value $F(x)$. Then in the set $Sel(F)$ of all continuous singlevalued selections of $F$ endowed with the fine topology the subset of those
selections of $F$ which pointwisely avoid values $\Psi(x)$ constitutes a dense $G_\delta$-subspace.
\end{theorem}

Recall that for a metric range space $(Y;d)$ the 
{\it fine} topology\index{fine topology}
in $C(X,Y)$ is defined by its local base
$$
V(f, \varepsilon(\cdot)) = \{g \in C(X,Y):\,\,d(g(x),f(x)) < \varepsilon(x),\,\,x \in X\},
$$
when $\varepsilon(\cdot)$ runs over the set of all positive continuous functions on $X$. Also, for metric space $(B;\varrho),$ a subset $A \subset B$ is said to be a  $\sigma Z_{n+1}$-subset if
it is the union of countably many
sets $A_i \subset B$ such that each continuous mapping from the $(n+1)$-cell to $B$ can be approximated
(with respect to the uniform topology)
by a sequence of continuous mappings to $B \setminus A_i$.  The sequential process of proving
this theorem is based on the following result.

\begin{theorem} [\cite{GVDen}]
Let all assumptions of Theorem 61 be true with exception that $\Psi$ is a closed-graph
mapping. Let $f$ be a continuous singlevalued selection of $F$, $\varepsilon(\cdot)$ a positive continuous function on $X$ and $\varrho$ a compatible metric on $Y$. Then
$F$ admits a continuous singlevalued selection $g$ such that $g(x) \notin \Psi(x)$ and
$\varrho(g(x),f(x)) < \varepsilon(x),$ for every $x \in X$.
\end{theorem}

Next, recall that Shchepin and Brodsky \cite{BrSc} proved that for any paracompact space $X$ with
$dimX \leq n+1$, a completely metrizable space $Y$, and for any $L$-filtration $\{F_i\}_{i=0}^{n+1}$  of maps $F_i:X \to Y,$
the ending mapping $F_{n+1}$ admits a continuous singlevalued selection.
One of the points in the definition of $L$-filtration (cf.  \cite{RS-2}), is the property
that graph-values $\{x\} \times F_{i}(x)$ are closed in some prescribed $G_\delta$-subset
of the product $X \times Y$. In \cite[Corollary 7.10]{GuFin} Gutev proved a generalization
of this result to the case when the graph of mapping $F_{n+1}$ is
a $G_\delta$-subset of
$X \times Y$. 
Roughly speaking, Gutev proposed his own version of the Shchepin-Brodsky 
$L$-filtrations technique. In particular, he generalized  the previous theorem
\cite[Corollary 7.12]{GuFin}.

\begin{theorem} [\cite{GuFin}]
Theorem 61 is true under the change of the assumption
$(\{F(x)\}_{x \in X}) 
\,\,
{\rm is}
\,\,
ELC^{n})$
with the restriction that $F$ be an 
$ELC^{n}$-mapping,
i.e.
${\{x\} \times \{F(x)\}}_{x \in X}) \,\,{\rm is}\,\,ELC^{n}$ in the product $X \times Y$.
\end{theorem}

As an example of a result on improvement of near-selections in \cite[Corollary 7.15]{GuFin} we
quote the following:
\begin{theorem} [\cite{GuFin}]
Let $X$ be a countably paracompact and normal space, $(Y;d)$ a complete metric space
and $F:X \to Y$ a Hausdorff continuous closedvalued mapping all of whose values $F(x)$ are uniformly
$LC^n$-subsets of $Y$. Then for every positive lower semicontinuous numerical function $\varepsilon(\cdot)$ on $X$ there exists a positive lower semicontinuous numerical function $\delta(\cdot)$ on $X$ with the following property:

If $g:X \to Y$ is a continuous singlevalued  $\delta$-selection of $F$ then there exists a continuous singlevalued
selection $f$ of $F$ such that
$d(g(x),f(x)) < \epsilon(x),\,\,x \in X$.

Moreover, if all values $F(x)$ are $n$-connected then $F$ admits a continuous singlevalued selection.
\end{theorem}

During preparation of the previous survey \cite{RS-2}, Brodsky, Shchepin and Chigogidze announced
results on problem of local triviality of Serre fibrations with two-dimensional fibers. Their paper \cite{BrScCh} appeared in 2008.

\begin{theorem}
Let $p:E \to B$ be a Serre fibration of a locally connected compact space $E$ onto a compact
$ANR$-space $B$. Let all fibers $p^{-1}(x)$ be homeomorphic to a fixed two-dimensional manifold
$M$ which differs from the sphere and the projective plane. Then $p$ admits a continuous section (i.e. $p^{-1}:B \to E$ admits a continuous singlevalued selection) provided that one of the following restrictions holds:
\begin{description}
\item{(1)}   $\pi_1(M)$ is abelian and $H^2(B;\pi_1(M))=0$;
\item{(2)}    $\pi_1(M)$ is nonabelian,  $M$ is not
the Klein bottle and $\pi_1(B)=0$;
\item{(3)}    $M$ is the Klein bottle and $\pi_1(B)= \pi_2(B)=0$.
\end{description}
\end{theorem}

The proof follows the strategy of \cite{BrSc} and is based on selection results for $L$-filtrations which, roughly speaking, are
realized by the interplay between $L$-theory and $U$-theory
of multivalued mappings (cf.  \cite{RS-2}).

\medskip
{\bf 7.6.} Various results exist in continuous selection theory with some additional restriction of algebraic nature. For example, a multivalued mapping $F:X \to Y$ between (as a rule, locally convex topological) vector spaces is called {\it additive} if $F(x_1 + x_2) = F(x_1) + F(x_2)$,
$x_1,x_2\in X$, i.e., the image of the sum of two points coincides with the Minkowski sum of the
images of these points. Next,
$F:X \to Y$ is called {\it subadditive} ( resp., {\it superadditive}, resp., {\it convex}) if $F(x_1 + x_2) \subset F(x_1) + F(x_2)$, ( resp., if $F(x_1 + x_2) \supset F(x_1) + F(x_2)$, resp., if $F(tx_1 + (1-t)x_2) \subset
tF(x_1)+ (1-t)F(x_2),\,\,t\in [0,1]$).
The existence of linear selections was
proved for several such types of mappings with compact convex values in locally convex spaces. In particular, additive mappings
always have linear selections \cite{Gor}, and
every superadditive mapping possesses a linear selection\,\cite{Sm2}. Recently, Protasov \cite{Pr} obtained criteria on $X$ and $Y$ for the affirmative answer on existence of linear selections for arbitrary subadditive mappings.

\begin{theorem}
\begin{description}
\item{(1)} Any subadditive mapping $F:X \to Y$ with compact convex values has a continuous linear selection if and only if $\dim X = 1$ or $\dim Y = 1$.
\item{(2)} Any convex mapping $F:X \to Y$ with compact convex values has a continuous affine selection if and only if $\dim Y = 1$.
\end{description}
\end{theorem}

Moreover, in "only if" \, parts of $(1)$ and $(2)$ one can omit the continuity restriction: if $\dim Y \geq 2$ then there exists a convex mapping $F:X \to Y$ without affine selections. Applications in Lipschitz stability problem for linear operators in Banach spaces are presented in \cite{Pr} as well.

\medskip
{\bf 7.7.} We end our survey by some selected results from metric projection theory.
For more detailed information cf.  \cite{BDIK, Esp}.

Let us recall that the {\it operator of almost best approximation}, or 
{\it $\varepsilon$-projection,}\index{$\varepsilon$-projection}
of a real Banach space $(X,\|\cdot\|_X)$ to
a set $M\subset X$ is defined as the multivalued map
$$
x\mapsto P_{M,\varepsilon}(x) = \{z\in M: \|z-x\|_X\le \rho(x, M) +\varepsilon\},
$$
where $\rho(x,M)=\inf_{y\in M}\|x-y\|_X$ is the distance from $x$ to $M$. If $\varepsilon=0$, then $P_{M}=P_{M, 0}$ is the metric projection operator. Clearly, all sets $P_{M,\varepsilon}(x)$
are nonempty, whereas the equality $P_M (x) = \emptyset$  is in general,  possible.

If  $\|z-x\|_X\le \rho(x, M) +\varepsilon$ is replaced by $\|z-x\|_X\le (1+\epsilon) \rho(x, M)$ then
such multivalued mapping is called a
{\it multiplicative $\varepsilon$-projection onto $M$}\index{multiplicative $\varepsilon$-projection}.

Recall the Konyagin theorem \cite{Ko} which states that the metric $\varepsilon$-projection operator admits a continuous singlevalued selection in the case $X=C[0;1]$ with the standard norm
and where $M = R =\{\frac{f}{g}:\,\,  f\in U,  g \in W\}$ is the set of all generalized fractions with
$U$ and $W$ linear closed subspaces of $X$. Note that this result is not true for $X=L_p[0;1]$.

Ryutin \cite{Ryu-1, Ryu-2} considered the set
$$
R_{U,W} = Clos\left\{\frac{f}{g}:\,\,  f\in U,  g \in W\ \,\,{\rm and}\,\, g > 0 \right\}
$$
of generalized fractions in the space $X=C(K),$  where $K$ is connected metric compact, or in the
the space $X=L_1[0;1]$, and with finite-dimensional $U$ and $W$.

\begin{theorem} [\cite{Ryu-2}]
Let the intersection of the set $R_{U,W}$ with the closed unit ball $D \subset C(K)$ be  compact.
Then for every
$\varepsilon >0$ the multiplicative $\varepsilon$-projection of $D$ onto $R_{U,W}$ admits a uniformly continuous singlevalued selection.
\end{theorem}

In the space $X=L_1[0;1]$ the situation is more complicated. Namely, in \cite{Ryu-1}
a wide class of pairs $(U,W)$ of finite-dimensional subspaces in $L_1[0;1]$ were found with
the property that there exists $\varepsilon_0=\varepsilon_{0}(U,W) > 0$ such that
the multiplicative $\varepsilon$-projection onto $R_{U,W}$ admits
a singlevalued continuous selection only if $\varepsilon \geq \varepsilon_0$.

Livshits \cite{Liv-1,Liv-2} considered $X=C[0;1]$ with the standard norm and
continuous selections of the metric $\varepsilon$-projection operator onto
the set of all {\it splines} (i.e. piecewise polynomials) with nonfixed nodes.
Namely, for a fixed $n,d\in \N,$ denote by $S_n^d[0,1]$
the set of all functions $f \in C[0;1]$ such that for some (depending on $f$) nodes
$0 = x_0<x_1 <\cdots<x_{n-1}<x_n=1,$ each restriction $f|_{[x_{k-1},x_k]}$  is a
polynomial of degree $\le d$.

\begin{theorem} [\cite{Liv-1,Liv-2}]
\begin{description}
 \item{(1)} A continuous singlevalued selection of the metric projection onto the set $S_n^1[0,1]$ exists if and only if $n \leq 2$;
\item{(2)} For any $\varepsilon > 0$ and any $n \in \N,$ there exists a continuous
singlevalued selection of the metric $\varepsilon$-projection onto the set $S_n^1[0,1]$;
\item{(3)}  For any  $ n >1$ and any $d>1$ there exists $\varepsilon = \varepsilon(n,d)$
such that there is no continuous singlevalued selection of
the metric $\varepsilon$-projection onto the set $S_n^d[0,1]$.
\end{description}
\end{theorem}

\section*{Acknowledgements}
The authors were supported by the SRA grants P1-0292-0101, J1-2057-0101 and J1-4144-0101, and the RFBR grant 11-01-00822.

\newcommand{\noopsort}[1]{} \newcommand{\singleletter}[1]{#1}

\printindex
\end{document}